\documentclass[10pt]{article}
\newcommand{\D}{\displaystyle}
\newcommand{\DF}[2]{\frac{\D#1}{\D#2}}
\usepackage[fleqn]{amsmath}
\usepackage{amssymb}
\usepackage{listings}
\usepackage{newlfont}
\usepackage{algorithm}
\usepackage{multirow}
\usepackage{graphicx}
\usepackage{subeqnarray}
\usepackage{cases}
\usepackage{latexsym,bm}
\usepackage{indentfirst}
\usepackage{hyperref}

\usepackage{epstopdf}

\newtheorem{lem}{Lemma}[section]
\newtheorem{theorem}{Theorem}[section]

\newtheorem{remark}{Remark}[section]

\numberwithin{equation}{section}
\newenvironment{proof}[1][Proof]{\textbf{#1.} }
{\ \rule{0.75em}{0.75em}\smallskip}

\DeclareMathOperator{\prox}{prox}

\date{}
\title{A primal-dual fixed-point algorithm for minimization of the sum of three convex separable functions}
\author{Peijun Chen$^{1,2,3}$,\  Jianguo Huang$^{1}$,\ Xiaoqun Zhang$^{1,4}$
\\
\\$^{1}${\small \textit{Department of Mathematics, Shanghai Jiao Tong University,}} \\{\small \textit{Shanghai 200240, China}}
\\
$^{2}${\small \textit{School of Biomedical Engineering, Shanghai Jiao Tong University,}}\\{\small \textit{Shanghai 200240, China}}
\\
$^{3}${\small \textit{ Department of Mathematics, Taiyuan University of Science and Technology,}}
\\{\small \textit{Taiyuan 030024, China}}
\\$^{4}$ {\small\textit{Institute of Natural Sciences, Shanghai Jiao Tong University, }}
\\
{\small\textit{Shanghai 200240, China}}
\\
{\small Email: chenpeijun@sjtu.edu.cn, jghuang@sjtu.edu.cn and xqzhang@sjtu.edu.cn}
}

\begin{document}\maketitle
\begin{abstract}
 Many problems arising in image processing and signal recovery with multi-regularization  can be formulated as minimization of a sum of three convex separable functions. Typically, the objective function involves a smooth function with Lipschitz continuous gradient, a linear composite nonsmooth function and a nonsmooth  function.  In this paper, we propose a primal-dual fixed-point (PDFP) scheme  to solve the above class of problems. The proposed algorithm for three block problems is a fully splitting symmetric scheme, only involving explicit gradient and linear operators without inner iteration, when the nonsmooth functions can be easily solved via their proximity
operators, such as $\ell_1$ type regularization. We study the convergence of the proposed algorithm and illustrate its efficiency through examples on fused LASSO  and image restoration with non-negative constraint and sparse regularization.
\end{abstract}


\section{Introduction}

In this paper, we aim to design a  primal-dual fixed-point algorithmic framework for solving the following minimization problem:
\begin{align}
   \underset{x\in \mathbb{R}^n}{\mbox{min}}\quad {f_1}(x)+({f_2}\circ B)(x)+{f_3}(x), \label{PDFP2O3B:eqbasic}
\end{align}
where ${f_1},{f_2}$ and ${f_3}$ are three proper lower semi-continuous convex functions,  and ${f_1}$ is differentiable on  ${\mathbb{R}^n}$ with a $1/\beta$-Lipschitz continuous gradient for some $\beta\in(0,+\infty]$, while $B:{\mathbb{R}^n}\rightarrow \mathbb{R}^m$ is a bounded linear transformation. This formulation covers a wide application in image processing and signal recovery with multi-regularity terms.
For instance, in many imaging and data processing applications, the functional $f_1$  corresponds
to a data-fidelity term, and the last two terms are related to regularity terms.
As a direct example of \eqref{PDFP2O3B:eqbasic}, we can consider  the fused LASSO penalized
problems \cite{YL06} defined  by
\begin{align*}
   \underset{x\in {\mathbb{R}^n}}{\mbox{min}}\quad \frac 1 2\|Ax-a\|^2+ \mu_1 \|Bx\|_1+\mu_2 \|x\|_1.
\end{align*}
On the other hand, in the imaging science, total variation regularization with $B$ being the discrete gradient  operator together with $\ell_1$ regularization has been adopted in some image restoration applications, for example in \cite{Goldstein2009}.

As far as we know, Condat \cite{C13} tackled a problem with the same form as given in  \eqref{PDFP2O3B:eqbasic} and proposed a primal dual splitting scheme. Extensions to multi-block composite functions are also discussed in detail.  For the special case $B=I$ ($I$ denotes the usual identity operator), Davis and Yin  \cite{DY15} proposed a three block operator splitting scheme based on monotone operators. When the problem \eqref{PDFP2O3B:eqbasic}  reduces to  two-block separable functions, many splitting and proximal algorithms have been proposed and studied in the literature. Among them, extensive research have been conducted on the  alternating direction of multiplier method (ADMM) \cite{Fortin1983} (also known as split Bregman \cite{Goldstein2009}, see for example \cite{Boyd2010} and the references therein).  The primal-dual hybrid gradient method (PDHG) \cite{Zhu2008,Esser2010,CP11,PC11}, also known as Chambolle-Pock algorithm \cite{CP11}, is another class of popular algorithm, largely adopted in imaging applications. In \cite{Zhang2010,CHZ13}, several completely decoupled schemes, such as inexact Uzawa solver and  primal-dual fixed-point algorithm, are proposed to avoid subproblem solving for some typical $\ell_1$ minimization problems. Komodakis and Pesquet \cite{KP14} recently gave a nice overview of recent primal-dual approaches for solving large-scale optimization problems \eqref{PDFP2O3B:eqbasic}.
A general class of multi-step fixed-point proximity algorithms is  proposed in \cite{LSXZ15}, which covers several existing algorithms \cite{CP11,PC11} as special cases.
In the preparation of this paper,  we notice that Li and Zhang \cite{LZ15}  also studied  the problem \eqref{PDFP2O3B:eqbasic} and introduced a quasi-Newton based scheme as preconditioned operators and the overrelaxation strategies to accelerate the algorithms. Both algorithms can be viewed as a generalization of  Condat's algorithm  \cite{C13}. The theoretical analysis is established  based on the multi-step techniques present in \cite{LSXZ15}.

In the following, we  mainly review some most relevant work for a concise presentation. We first consider a constrained regularization problem
 \begin{align}
   \underset{x\in C}{\mbox{ min}}\quad  {f_1}(x)+({f_2}\circ B)(x), \label{PDFP2OC:eqbasic}
\end{align}
where $C \subset \mathbb{R}^n$ is a closed convex set arising from physical requirements of the solutions.  This problem can be reformulated as the form \eqref{PDFP2O3B:eqbasic} by introducing a set indicator function $\chi_C$ (see \eqref{notation:chiC}) as $f_3$.  For example, this problem \eqref{PDFP2OC:eqbasic} has been studied in \cite{KLSX12} in the context of maximum a posterior ECT reconstruction, and  a preconditioned alternating projection algorithm (PAPA) is proposed for solving the resulted regularization problem. For $f_3=0$ in \eqref{PDFP2O3B:eqbasic}, we proposed a primal-dual fixed-point algorithm PDFP$^2$O (primal-dual fixed-point algorithm based on proximity operator) in \cite{CHZ13}. Based on the fixed point theory, we have shown the convergence of the scheme PDFP$^2$O and its convergence rate under suitable conditions.

In this work, we aim to extend the ideas of PDFP$^2$O in \cite{CHZ13} and PAPA in \cite{KLSX12} for solving \eqref{PDFP2O3B:eqbasic} without subproblem solving and provide a convergence analysis on the primal dual sequences.  The specific algorithm, namely  primal-dual fixed-point (PDFP) algorithm, is formulated as follows:
\begin{subequations}
 \label{formbasic3B}
 \begin{numcases}{(\mbox{PDFP}) \quad }
   y^{k+1}=\prox_{{\gamma}{{f_3}}}(x^k-\gamma\nabla {{f_1}}(x^k)-{\lambda} B^T  v^{k}),\label{formbasic3Ba}\\
   v^{k+1}=(I-\prox_{\frac{\gamma}{\lambda}{{f_2}}})(By^{k+1}+v^k)\label{formbasic3Bb},\\
   x^{k+1}=\prox_{{\gamma}{{f_3}}}(x^k-\gamma\nabla {{f_1}}(x^k)-{\lambda} B^T  v^{k+1}),\label{formbasic3Bc}
 \end{numcases}
\end{subequations}
where $0<\lambda< 1/\lambda_{\max}(BB^T)$, $0<\gamma<2\beta$. Here $\prox_f$  is the proximity operator \cite{M62} of a function $f$, see \eqref{notation:prox}.
When $f_3=\chi_C$,  the proposed algorithm \eqref{formbasic3B} is reduced to  PAPA proposed in \cite{KLSX12}, see \eqref{formbasicC2}. For the special case $f_3=0$, we obtain PDFP$^2$O proposed in \cite{CHZ13}.  The convergence analysis of the proposed PDFP algorithm is  built upon fixed point theory on the primal and dual pairs.
The overall scheme is completely explicit, which allows an easy implementation and parallel computing for many large scale applications. This will be further illustrated through application to the problems arising in statistics learning and image restoration.
The PDFP is a symmetric  form  and it is  different from Condat's algorithm proposed in \cite{C13}. In addition, we point out that  the  ranges of the parameters are larger  than those of \cite{C13,LZ15} and may lead to a significant advantage of parameter selection in practice. This will be further discussed in section \ref{seq:Connections}.

The rest of the paper is organized as follows.  In section \ref{seq:Derivation}, we will present some preliminaries and notations, and  deduce {PDFP} from the first order optimality condition. In section \ref{seq:Convergence}, we will provide the convergence results and the linear convergence rate results for some special cases. In section \ref{seq:Connections},  we will make a comparison on the form of  the PDFP algorithm \eqref{formbasic3B} with some existing algorithms. In section \ref{sec:Numerical_experiments}, we will show the numerical performance and the efficiency of {PDFP} through some examples on fused LASSO and pMRI (parallel magnetic resonance image) reconstruction.

\section{Primal dual fixed point algorithm}\label{seq:Derivation}
\subsection{Preliminaries and notations}\label{seq:Preparation}
For the self completeness of this work, we list  some relevant notations,  definitions, assumption and lemmas in convex analysis. One may refer to \cite{CW05,CHZ13} and the references therein for more details.

For the ease of presentation, we restrict our discussion in the Euclidean space $\mathbb{R}^n$, equipped with the usual inner product $\langle \cdot,\cdot \rangle$  and
norm $\|\cdot\|=\langle \cdot,\cdot \rangle^{1/2}$. We first assume that the problem \eqref{PDFP2O3B:eqbasic} has at least one solution and ${f_2},\ {f_3},\ B$ satisfy
\begin{align}
   0\in \mbox{int}(\mbox{dom}_{{f_2}}- B(\mbox{dom}_{{f_3}})),\label{asumption01}
\end{align}
where the symbol $\mbox{int}(\cdot)$ denotes the strong relative interior of a convex subset, and the effective domain of $f$ is defined as
$
   \mbox{dom}_f = \{x\in \mathbb{R}^n| f(x) < +\infty\}.
$

The $\ell_1$ norm of a vector $x\in \mathbb{R}^n$ is denoted by $\|\cdot\|_1$ and  the spectral norm of a matrix is denoted by $\|\cdot\|_2$. Let $\Gamma_0(\mathbb{R}^n)$ be the collection of all proper lower semi-continuous convex functions from $\mathbb{R}^n$ to $(-\infty,+\infty]$.
For a function $f\in\Gamma_0(\mathbb{R}^n)$,  the proximity operator of $f$: $\prox_f$ \cite{M62} is defined by
\begin{align}
   \prox_{f}(x)= \underset{y\in\mathbb{R}^n}{\mbox{arg min}}\ {f
    (y)+\DF{1}{2}\|x-y\|^2}. \label{notation:prox}
\end{align}
For a nonempty closed convex set $C\subset \mathbb{R}^n$, let $\chi_C$  be the indicator function of $C$, defined by
\begin{align}
   \chi_C(x)=
     \left\{
      \begin{array}{ll}
         0, &x\in C,\\
         +\infty,&x\not\in C.
      \end{array}
     \right.\label{notation:chiC}
\end{align}
Let $\mbox{proj}_C$ be the projection operator onto $C$, i.e.
\begin{align*}
   \mbox{proj}_{C}(x)= \underset{y\in C}{\mbox{arg min}}\  \|x-y\|^2. 
\end{align*}
It is easy to see that $\prox_{\gamma \chi_C}=\mbox{proj}_C$ for all $\gamma>0$, and the proximity operator is a generalization of projection operator.  Note that many efficient splitting algorithms rely on the fact that $\prox_{f}$ has a closed form solution. For example, when $f=\gamma\|\cdot\|_1$, the proximity solution is given by element-wised soft-shrinking. We refer the reader to \cite{CW05} for more details about proximity operators.
Let  $\partial f$  be the subdifferential of $f$, i.e.
\begin{align}
      \partial f(x)=\{v\in\mathbb{R}^n\ |\ \langle y-x,v\rangle\leq f(y)-f(x)\ \mbox{ for all } y\in\mathbb{R}^n\}, \label{partialf}
\end{align}
and ${f^*}$ be the convex conjugate function  of ${f}$, defined by
\begin{align*}
      f^*(x)=\underset{y \in \mathbb{R}^n}{\sup}\langle x,y\rangle-f(y).
\end{align*}

An operator $T:\mathbb{R}^n\rightarrow\mathbb{R}^n$ is nonexpansive if
\begin{align*}
      \|Tx-Ty\|\leq\|x-y\|\ \mbox{ for all }  x,y \in\mathbb{R}^n,
\end{align*}
and $T$ is firmly nonexpansive if
\begin{align*}
   \ \|Tx-Ty\|^2\leq\langle Tx-Ty,x- {y}\rangle \ \mbox{ for all }  x,y \in\mathbb{R}^n.
\end{align*}
It is obvious that a firmly nonexpansive operator is nonexpansive.
An operator $T$ is $\delta$-strongly monotone if there exists a positive real number $\delta$ such that
\begin{align}
   \label{definition:stronglymonotone}
     \langle Tx-Ty,x- {y}\rangle\geq \delta\|x-y\|^2 \ \mbox{ for all }  x,y \in\mathbb{R}^n.
\end{align}

\begin{lem}
   \label{lem_additivity}
   For any two functions ${{f_2}} \in \Gamma_0 (\mathbb{R}^m)$ and ${{f_3}}\in\Gamma_0 (\mathbb{R}^n)$, and a bounded linear transformation $B:{\mathbb{R}^n}\rightarrow \mathbb{R}^m$, satisfying that
   $
      0\in \mbox{int}(\mbox{dom}_{{f_2}}- B(\mbox{dom}_{{f_3}})),
   $
   there holds
     \[
          \partial ({f_2}\circ B+{f_3})=B^T\circ\partial{{f_2}}\circ B+\partial {f_3}.
   \]
\end{lem}
\begin{lem}
   \label{lem_proximity}
   Let ${f}\in \Gamma_0(\mathbb{R}^n)$. Then $\prox_f$ and $I-\prox_f$ are firmly nonexpansive.
   In addition, there hold
   \begin{align}
      &x=\prox_{{f}}(y)  \Leftrightarrow y-x\in \partial{{f}}(x) \mbox{ for a given  } y\in \mathbb{R}^n,\label{lem_Fema}\\ %
      &y\in \partial{{f}}(x) \Leftrightarrow  x=\prox_{{f}}(x+y)\nonumber\\
      &\phantom{y\in \partial{{f}}(x)}\Leftrightarrow y=(I-\prox_{{f}})(x+y)\mbox{ for } x,y\in \mathbb{R}^n, \label{lem_keylem}\\
      &x=\prox_{\gamma {f}}(x)+ \gamma \prox_{\frac{1}{\gamma}{f^*}}({\frac{1}{\gamma} x})  \mbox{  for all }  x\in \mathbb{R}^n \mbox{ and }\gamma>0.
      \label{lem_Moreau_decomposition}
   \end{align}
   If $f$ has $1/\beta$-Lipschitz continuous gradient further, there holds
   \begin{align}\label{lem_LCG}
      \beta\|\nabla f(x)-\nabla f(y)\|^2\leq \langle \nabla f(x)-\nabla f(y), x-y\rangle \ \mbox{ for all } x, y \in\mathbb{R}^n.
   \end{align}
\end{lem}
\begin{lem}\label{lem:convergence}
Let $T$ be an operator, and $u^*$ be a fixed-point of $T$. Let $\{u^{k}\}$ be the sequence generated by the fixed point iteration $u^{k+1}=T(u^{k})$. Suppose (i) $T$ is continuous,  (ii) $\{\|u^k-u^*\|\}$  is non-increasing, (iii) $\lim_{k\to +\infty}  \|u^{k+1}-u^k\|=0$. Then the sequence $\{u^{k}\}$ is bounded and converges  to a fixed-point of $T$.
\end{lem}
The proof of Lemma \ref{lem:convergence} is standard, and one may refer to  the proof of Theorem 3.5 in \cite{CHZ13} for more details.

Let $\gamma$ and $\lambda$ be two positive numbers. To simplify the presentation, we use in what follows the following notations:
\begin{align}
      &T_0(v,x)=\prox_{{\gamma}{{f_3}}}(x-\gamma\nabla {{f_1}}(x)-{\lambda} B^T v), \label{notation_Ty}\\
      &T_1(v,x)=(I-\prox_{\frac{\gamma}{\lambda}{{f_2}}})(B\circ T_0(v,x)+v),\label{notation_TvC}\\
      &T_2(v,x)=\prox_{{\gamma}{{f_3}}}(x-\gamma\nabla {{f_1}}(x)-{\lambda} B^T \circ T_1(v,x)),\label{notation_Tx}\\
      &T(v,x)=
       \left(
          {T_1}(v,x),\
          {T_2(v,x)}
       \right).\label{notation_T}
\end{align}
Denote
\begin{align}
   &g(x)=x-\gamma \nabla f_1(x)\  \mbox{ for all } x\in \mathbb{R}^n, \label{notation_g}\\
   &M=I-\lambda BB^T.\label{notation_M}
\end{align}
When $0<\lambda< 1/\lambda_{\max}(BB^T)$, $M$ is a positive symmetric definite matrix, so we can define a norm
\begin{align}
   \|v\|_{M}=\sqrt{\langle{v},Mv\rangle}\ \mbox{ for all } v\in \mathbb{R}^m. \label{notation_M_norm}
\end{align}
For a pair ${u}=\left(v,x\right)\in \mathbb{R}^m\times\mathbb{R}^n$, we also define
a norm on the product space $\mathbb{R}^m\times\mathbb{R}^n$ as
\begin{align}
   \|u\|_{\lambda}=\sqrt{\lambda\|v\|^2+\|x\|^2}. \label{notation_lambda_norm}
\end{align}

\subsection{Derivation of {PDFP}}

On extending the ideas of  PAPA proposed in \cite{KLSX12} and PDFP$^2$O proposed in \cite{CHZ13}, we derive the  following  primal-dual fixed-point algorithm \eqref{formbasic3B} for solving the minimization problem \eqref{PDFP2O3B:eqbasic}.

Under the assumption \eqref{asumption01}, by using the first order optimality condition of \eqref{PDFP2O3B:eqbasic} and lemma \ref{lem_additivity}, we have
\begin{align*}
    0\in \gamma\nabla {{f_1}}(x^*)+{\lambda}B^T\partial(\DF{\gamma}{\lambda} {{f_2}})(Bx^*)+\gamma\partial {f_3}(x^*),
\end{align*}
where $x^*$ is an optimal solution.  Let
\begin{align}
   v^*\in \partial(\DF{\gamma}{\lambda} {{f_2}})(Bx^*).
	\label{eq:v}
\end{align}
By applying 
\eqref{lem_keylem}, we have
\begin{align}
   &v^*=
   (I-\prox_{\frac{\gamma}{\lambda} {{f_2}}})(Bx^*+v^*), \label{eqkey3}\\
   &x^*=\prox_{ {\gamma}  {{f_3}}}(x^*-\gamma\nabla {{f_1}}(x^*)-\lambda B^Tv^*). \label{eqkey2}
\end{align}
By inserting \eqref{eqkey2} into \eqref{eqkey3}, we get
\[
   v^*=(I-
   \prox_{\frac{\gamma}{\lambda}{{f_2}}})(B\circ \prox_{{\gamma}  {{f_3}}}(x^*-\gamma\nabla {{f_1}}(x^*)-\lambda B^Tv^*)+v^*),
\]
or equivalently, $v^*=T_1(v^*,x^*)$. Next, replacing $v^*$ in \eqref{eqkey2} by $T_1(v^*,x^*)$, we can get $x^*=T_2(v^*,x^*)$.
In other words $u^*=T(u^*)$  for ${u^*}=\left(v^*,x^*\right)$.
Meanwhile, if $u^*=T(u^*)$, we can get that $x^*$ meets the first order optimality condition of \eqref{PDFP2O3B:eqbasic} and thus $x^*$  is a minimizer of \eqref{PDFP2O3B:eqbasic}.

To sum up, we have the following theorem.

\begin{theorem}\label{theorem_fixedpoint}
  Suppose that $x^*$ is a solution of \eqref{PDFP2O3B:eqbasic} and $v^*\in \mathbb{R}^m$ is defined as \eqref{eq:v}. Then we have
   \begin{subequations}
     \nonumber
     \begin{numcases}{}
       v^*=T_1(v^*,x^*),\nonumber\\
       x^*=T_2(v^*,x^*),\nonumber
     \end{numcases}
   \end{subequations}
   i.e. ${u^*}=\left({v^*},{x^*}\right)$ is a fixed-point of $T$. Conversely, if ${u^*}= \left({v^*},{x^*}\right)\in \mathbb{R}^m\times\mathbb{R}^n$ is a fixed-point of $T$, then $x^*$  is a solution of \eqref{PDFP2O3B:eqbasic}.
\end{theorem}

It is easy to confirm that the sequence $\{(v^{k+1}, x^{k+1})\}$ generated by {PDFP} algorithm \eqref{formbasic3B} is the Picard iteration
$(v^{k+1},x^{k+1})={T}(v^k,x^k)$. So we will use the operator $T$ to analyze the convergence of {PDFP} in section \ref{seq:Convergence}.


\section{Convergence analysis}\label{seq:Convergence}
In the following, we denote ${u^*}=\left({v^*},{x^*}\right)$ as a fixed-point of the operator $T$.
Let $\{u^{k+1}=(v^{k+1},x^{k+1})\}$ be the sequence generated by the operator $T$.
\subsection{Convergence}
\begin{lem} { There hold the following estimates:}
 \begin{align}
   \|v^{k+1}-v^*\|^2&\leq \|v^{k}-v^*\|^2-\|v^{k+1}-v^k\|^2+2 \langle B^T(v^{k+1}-v^*),y^{k+1}-x^*\rangle,\label{eq:verror}\\
   \|x^{k+1}-x^*\|^2&\leq\|x^{k}-x^*\|^2-\|x^{k+1}-y^{k+1}\|^2-\|x^k-y^{k+1}\|^2\nonumber\\
	&+2\langle  x^{k+1}-y^{k+1}, \gamma\nabla {f_1}(x^k)+\lambda B^Tv^k \rangle\nonumber \\
   &-2\langle  x^{k+1}-x^*, \gamma\nabla {f_1}(x^k)+\lambda B^Tv^{k+1} \rangle+2\gamma ({f_3}(x^*)-{f_3}(y^{k+1})).\label{eq:xerror}
\end{align}
\end{lem}
\begin{proof}
We first prove \eqref{eq:verror}.
By Lemma \ref{lem_proximity}, we know $I-\prox_{\frac{\gamma}{\lambda}{{f_2}}}$ is firmly nonexpansive,  and use \eqref{formbasic3Bb} and \eqref{eqkey3} we further have
\begin{align*}
    \|v^{k+1}-v^*\|^2\leq \langle v^{k+1}-v^*,(B y^{k+1}+v^k)-(Bx^*+v^*)\rangle,
\end{align*}
which implies
\begin{align*}
    \langle v^{k+1}-v^*,v^{k+1}-v^k\rangle &\leq \langle v^{k+1}-v^*,B (y^{k+1}-x^*)\rangle=\langle B^T (v^{k+1}-v^*),y^{k+1}-x^*\rangle.
\end{align*}
Thus
\begin{align*}
&\|v^{k+1}-v^*\|^2=\|v^{k}-v^*\|^2-\|v^{k+1}-v^k\|^2+2\langle v^{k+1}-v^*, v^{k+1}-v^k\rangle\\
\leq &\|v^{k}-v^*\|^2-\|v^{k+1}-v^k\|^2+2 \langle B^T(v^{k+1}-v^*),y^{k+1}-x^*\rangle.
\end{align*}

Next we prove \eqref{eq:xerror}. By the optimality condition of \eqref{formbasic3Bc} (cf. \eqref{lem_Fema}), we have
\begin{align*}
   (x^k-\gamma\nabla {f_1}(x^k)-\lambda B^Tv^{k+1})-x^{k+1} \in \gamma\partial {f_3}(x^{k+1}).
\end{align*}
By the property of subdifferentials (cf. \eqref{partialf}),
\begin{align*}
   \langle x^*-x^{k+1},(x^k-\gamma\nabla {f_1}(x^k)-\lambda B^Tv^{k+1})-x^{k+1} \rangle\leq \gamma ({f_3}(x^{*})- {f_3}(x^{k+1})),
\end{align*}
i.e.,
\begin{align*}
   &\langle x^{k+1}-x^*, x^{k+1}-x^k\rangle \leq -\langle  x^{k+1}-x^*, \gamma\nabla {f_1}(x^k)+\lambda B^Tv^{k+1} \rangle+\gamma ({f_3}(x^{*})- {f_3}(x^{k+1})).
\end{align*}
Therefore,
\begin{align}
   &\|x^{k+1}-x^*\|^2=\|x^{k}-x^*\|^2-\|x^{k+1}-x^k\|^2+2\langle x^{k+1}-x^*, x^{k+1}-x^k\rangle\nonumber\\
   \leq&\|x^{k}-x^*\|^2-\|x^{k+1}-x^k\|^2-2\langle  x^{k+1}-x^*, \gamma\nabla {f_1}(x^k)+\lambda B^Tv^{k+1} \rangle\nonumber\\
   &+2\gamma ({f_3}(x^{*})- {f_3}(x^{k+1})).\label{eq:xerrortemp}
\end{align}
On the other hand, by the optimality condition of \eqref{formbasic3Ba}, it follows that
\begin{align*}
   (x^k-\gamma\nabla {f_1}(x^k)-\lambda B^Tv^k)-y^{k+1} \in \gamma\partial {f_3}(y^{k+1}).
\end{align*}
Thanks to the property of subdifferentials, there holds
\begin{align*}
   \langle x^{k+1}-y^{k+1},(x^k-\gamma\nabla {f_1}(x^k)-\lambda B^Tv^k)-y^{k+1} \rangle\leq \gamma ({f_3}(x^{k+1})-{f_3}(y^{k+1})).
\end{align*}
So
\begin{align*}
   &\langle x^{k+1}-y^{k+1}, x^k-y^{k+1}\rangle\leq \langle  x^{k+1}-y^{k+1}, \gamma\nabla {f_1}(x^k)+\lambda B^Tv^k \rangle+\gamma ({f_3}(x^{k+1})-{f_3}(y^{k+1})).
\end{align*}
Thus
\begin{align*}
   &-\|x^{k+1}-x^k\|^2=-\|x^{k+1}-y^{k+1}\|^2-\|x^k-y^{k+1}\|^2+2\langle x^{k+1}-y^{k+1}, x^k-y^{k+1}\rangle\\
   \leq&-\|x^{k+1}-y^{k+1}\|^2-\|x^k-y^{k+1}\|^2+2\langle  x^{k+1}-y^{k+1}, \gamma\nabla {f_1}(x^k)+\lambda B^Tv^k \rangle\\
	&+2\gamma ({f_3}(x^{k+1})-{f_3}(y^{k+1})).
\end{align*}
Replacing the term $-\|x^{k+1}-x^k\|^2$ in \eqref{eq:xerrortemp} with the
right side term of the above inequality, we immediately obtain \eqref{eq:xerror}.
\end{proof}

\begin{lem}
\label{lem:sequencekeyineq} There holds
\begin{align}
\|u^{k+1}&-u^*\|_\lambda^2
\leq
\|u^k-u^*\|_\lambda^2-\lambda\|v^{k+1}-v^{k}\|^2_M-\|x^{k+1}-y^{k+1}+\lambda B^T(v^{k+1}-v^k)\|^2 \nonumber\\
 &-\|(x^k-y^{k+1})-(\gamma\nabla {f_1}(x^k)-\gamma\nabla {f_1}(x^*))\|^2-\gamma (2\beta- {\gamma} )\|\nabla {f_1}(x^k)-\nabla {f_1}(x^*)\|^2.
\label{keyineq01}
\end{align}
\end{lem}
\begin{proof}
Summing the two inequalities \eqref{eq:verror} and \eqref{eq:xerror} and re-arranging the terms, we have
\begin{align}
&\lambda\|v^{k+1}-v^*\|^2+\|x^{k+1}-x^*\|^2\nonumber\\
\leq &\lambda\|v^{k}-v^*\|^2+\|x^{k}-x^*\|^2-\lambda\|v^{k+1}-v^{k}\|^2-\|x^{k+1}-y^{k+1}\|^2-\|x^k-y^{k+1}\|^2\nonumber\\
&+2 \langle\lambda B^T (v^{k+1}-v^*),y^{k+1}-x^*\rangle+2\langle  x^{k+1}-y^{k+1}, \gamma\nabla {f_1}(x^k)+\lambda B^Tv^k \rangle\nonumber\\
&-2\langle  x^{k+1}-x^*, \gamma\nabla {f_1}(x^k)+\lambda B^Tv^{k+1}\rangle+2\gamma ({f_3}(x^*)-{f_3}(y^{k+1}))\nonumber\\
= &\lambda\|v^{k}-v^*\|^2+\|x^{k}-x^*\|^2-\lambda\|v^{k+1}-v^{k}\|^2-\|x^{k+1}-y^{k+1}\|^2-\|x^k-y^{k+1}\|^2\nonumber\\
  &+2 \langle\lambda B^T (v^{k+1}-v^k),y^{k+1}-x^{k+1}\rangle+2\langle x^k -y^{k+1},\gamma\nabla {f_1}(x^k)-\gamma\nabla {f_1}(x^*)\rangle\nonumber\\
 &-2\langle x^k-x^*,\gamma\nabla {f_1}(x^k)-\gamma\nabla {f_1}(x^*)\rangle\nonumber\\
&+2(\langle y^{k+1}-x^*,-\gamma \nabla {f_1}(x^*)-\lambda B^T v^* \rangle+\gamma ({f_3}(x^*)-{f_3}(y^{k+1})))\nonumber\\
=&\lambda\|v^{k}-v^*\|^2+\|x^{k}-x^*\|^2-\lambda\|v^{k+1}-v^{k}\|_M^2-\|x^{k+1}-y^{k+1}+\lambda B^T(v^{k+1}-v^k)\|^2\nonumber\\
 &-\|(x^k-y^{k+1})-(\gamma\nabla {f_1}(x^k)-\gamma\nabla {f_1}(x^*))\|^2+\|\gamma\nabla {f_1}(x^k)-\gamma\nabla {f_1}(x^*)\|^2\label{midkeyneq} \nonumber\\
 &-2\langle x^k-x^*,\gamma\nabla {f_1}(x^k)-\gamma\nabla {f_1}(x^*)\rangle
\nonumber\\
&+2(\langle y^{k+1}-x^*,-\gamma \nabla {f_1}(x^*)-\lambda B^T v^* \rangle+\gamma ({f_3}(x^*)-{f_3}(y^{k+1}))),
\end{align}
where $\|\cdot\|_M$ is given in \eqref{notation_M} and \eqref{notation_M_norm}.
Meanwhile, by the optimality condition of \eqref{eqkey2}, we have
\begin{align*}
  -\gamma\nabla {f_1}(x^*)-\lambda B^Tv^*\in \gamma\partial {f_3}(x^*),
\end{align*}
which implies
\begin{align}
   \langle y^{k+1}-x^*,-\gamma\nabla {f_1}(x^*)-\lambda B^Tv^*\rangle+\gamma ({f_3}(x^{*})- {f_3}(y^{k+1}))\leq 0.
	\label{xvcombine}
\end{align}
On the other hand, it follows from {\eqref{lem_LCG}} that
\begin{align}
   -\langle x^k- x^*,\nabla {f_1}(x^k)-\nabla {f_1}(x^*)\rangle\leq -{\beta}\|\nabla {f_1}(x^k)-
   \nabla {f_1}(x^*)\|^2.
   \label{eqbasicinequality02}
\end{align}
Recalling \eqref{notation_lambda_norm}, we immediately obtain \eqref{keyineq01} in terms of \eqref{midkeyneq}-\eqref{eqbasicinequality02}.
\end{proof}

\begin{lem}
\label{lem:sequence}
Let $0<\lambda<1/{\lambda_{\max}(BB^T)}$ and $0<\gamma<2\beta$. Then {the sequence $\{\|u^{k}-u^*\|_\lambda\}$ is non-increasing} and
$
\lim_{k\to +\infty}  \|u^{k+1}-u^k\|_\lambda=0.
$
\end{lem}
\begin{proof}
If $0<\lambda< 1/{\lambda_{\max}(BB^T)}$ and $0<\gamma<2\beta$, it follows from \eqref{keyineq01} that $\|u^{k+1}-u^*\|_\lambda\leq \|u^{k}-u^*\|_\lambda$, i.e. the sequence $\{\|u^{k}-u^*\|_\lambda\}$ is non-increasing. Moreover, summing  the inequalities \eqref{keyineq01} from $k=0$ to $k=+\infty$, we get
\begin{align}
   &\lim_{k\to +\infty}   \|v^{k+1}-v^k\|_M=0, \label{eqTconvergence05}\\
   &\lim_{k\to +\infty}   \|x^{k+1}-y^{k+1}+\lambda B^T(v^{k+1}-v^k)\|=0,\label{eqTconvergence04}\\
   &\lim_{k\to +\infty}   \|(x^k-y^{k+1})-(\gamma\nabla {f_1}(x^k)-\gamma\nabla {f_1}(x^*))\|=0,\label{eqTconvergence06}\\
   &\lim_{k\to +\infty}   \|\nabla {f_1}(x^k)-\nabla {f_1}(x^*)\|=0.\label{eqTconvergence07}
\end{align}
The combination of \eqref{eqTconvergence06} and \eqref{eqTconvergence07} gives
\begin{align}
   \lim_{k\to +\infty}  \|x^k-y^{k+1}\|=0. \label{eqTconvergence0c}
\end{align}
Noting that $0<\lambda<1/{\lambda_{\max}(BB^T)}$, we know $M$ is  positive symmetric definite, so \eqref{eqTconvergence05} is equivalent to
\begin{align}
   \lim_{k\to +\infty}  \|v^{k+1}-v^k\|=0. \label{eqTconvergence0b}
\end{align}
Hence, we have from the above inequality and \eqref{eqTconvergence04} that
\begin{align}
   \lim_{k\to +\infty}  \|x^{k+1}-y^{k+1}\|=0. \label{eqTconvergence0d}
\end{align}
The combination of \eqref{eqTconvergence0c} and  \eqref{eqTconvergence0d} then gives rise to
\begin{align}
   \lim_{k\to +\infty}   \|x^{k+1}-x^{k}\|=0. \label{eqTconvergence0a}
\end{align}
According to \eqref{eqTconvergence0b}, \eqref{eqTconvergence0a} and \eqref{notation_lambda_norm},  we have
$
\lim_{k\to +\infty}  \|u^{k+1}-u^k\|_\lambda=0.
$
\end{proof}

As a direct consequence of Lemma \ref{lem:sequence} and Lemma \ref{lem:convergence}, we obtain the convergence of {PDFP} as follows.
\begin{theorem}\label{theorem:convergence}
Let $0<\lambda<1/{\lambda_{\max}(BB^T)}$ and $0<\gamma<2\beta$. Then the sequence $\{u^{k}\}$ is bounded and converges  to a fixed-point of $T$, and  $\{x^{k}\}$ converges to a solution of \eqref{PDFP2O3B:eqbasic}.
\end{theorem}
\begin{proof}
By Lemma \ref{lem_proximity}, both $\prox_{{\gamma}{f_3}}$ and $I-\prox_{\frac{\gamma}{\lambda}{f_1}}$ are firmly nonexpansive, thus the operator $T$ defined by \eqref{notation_Ty}-\eqref{notation_T} is continuous.
From Lemma \ref{lem:sequence}, we know that the sequence  $\{\|u^k-u^*\|_\lambda\}$ is non-increasing and $\lim_{k\to +\infty}  \|u^{k+1}-u^k\|_\lambda=0$.
By using Lemma \ref{lem:convergence}, we know that the sequence $\{u^{k}\}$ is bounded and converges  to a fixed-point of $T$.
By using Theorem \ref{theorem_fixedpoint}, $\{x^{k}\}$ converges to a solution of \eqref{PDFP2O3B:eqbasic}.
\end{proof}

\begin{remark}\label{remark:thereason}
For the special case $f_3=0$, the PDFP reduces naturally to PDFP$^2$O \eqref{formbasicPDFP2O} proposed in \cite{CHZ13},
where the conditions for the parameters are $0<\lambda\leq 1/\lambda_{\max}(BB^T)$, $0<\gamma<2\beta$. In Theorem \ref{theorem:convergence}, the condition for the parameter $\lambda$ is slightly more restricted as $0<\lambda< 1/\lambda_{\max}(BB^T)$. It is easy to see when $f_3=0$, the equation  \eqref{keyineq01} in Lemma \ref{lem:sequencekeyineq}  reduces to
\begin{align*}
\|u^{k+1}-u^*\|_\lambda^2
\leq&
\|u^k-u^*\|_\lambda^2-\lambda\|v^{k+1}-v^{k}\|^2_M-\|\lambda B^T(v^k-v^*)\|^2\nonumber\\
&-\gamma (2\beta- {\gamma} )\|\nabla {f_1}(x^k)-\nabla {f_1}(x^*)\|^2,
\end{align*}
and the conditions in the proof of Lemma \ref{lem:sequence} can be relaxed to $0<\lambda\leq 1/\lambda_{\max}(BB^T)$. However, for a general $f_3$, the condition can not be relaxed to ensure the positive definitiveness of the matrix $M$, which is needed for the uniform convergence.
\end{remark}

\begin{remark}\label{remark:f10}
For the special case $f_1=0$, the problem \eqref{PDFP2O3B:eqbasic} reduces to two-block  proper lower semicontinuous convex functions without Lipschitz continuous gradient assumptions. The condition $0<\gamma<2\beta$ in PDFP becomes $0<\gamma<+\infty$. Although, $\gamma$ is an arbitrary positive number in theory, but the range of $\gamma$ will affect the convergence speed and it is also a difficult problem to choose a best value in practice.
\end{remark}

\subsection{Linear convergence rate for special cases}
\label{seq:Convergence_Rate}
In the following, we will show the convergence rate results with some additional assumptions on the  basic problem \eqref{PDFP2O3B:eqbasic}.
In particular, for $f_3=0$, the algorithm reduces to PDFP$^2$O proposed in \cite{CHZ13}. The conditions for a linear convergence  given there as   condition 3.1 in \cite{CHZ13} is as followed:    for $0<\lambda\leq 1/\lambda_{\max}(BB^T)$ and $0<\gamma<{2}\beta$, there exist $\eta_1$, $\eta_2\in [0,1)$ such that
\begin{align}
    &\|I-\lambda BB^T\|_2\leq\eta_1^2,\nonumber\\
    &\|g(x)-g(y)\|\le \eta_2 \|x-y\|\mbox{ for all } x,\ y\in \mathbb{R}^n,\label{condition02}
\end{align}
where 
$g(x)$ is given in \eqref{notation_g}.
It is easy to see that a strongly convex function  $f_1$ satisfies the condition \eqref{condition02}.  For a general $f_3$, we need stronger conditions on the functions.
\begin{theorem}\label{prop_contraction_ratio}
   Suppose  that \eqref{condition02} holds and  $f_2^*$ is strongly convex. Then, we have \begin{align*}
   \|u^{k+1}-u^*\|_{(1+\lambda\delta/\gamma)\lambda}
   \leq \eta \|u^k-u^*\|_{(1+\lambda\delta/\gamma)\lambda},
\end{align*}
where $0<\eta<1$ is the convergence rate (indicated in the proof) and $\delta>0$ is a parameter describing  the strongly monotone property of $\partial f_2^*$ (cf. \eqref{definition:stronglymonotone}).
\end{theorem}
\begin{proof}
Use Moreau's identity (cf. \eqref{lem_Moreau_decomposition}) to get
\[
   (I-\prox_{\frac{\gamma}{\lambda}{f_2}})(By^{k+1}+v^k)
   =\frac{\gamma}{\lambda}\prox_{\frac{\lambda}{\gamma}{f_2^*}}
   (\frac{\lambda}{\gamma}By^{k+1}+\frac{\lambda}{\gamma}v^k).
\]
So \eqref{formbasic3Bb} is equivalent to
\begin{align}\label{formbasic3Bb02}
\frac{\lambda}{\gamma}v^{k+1}
   =\prox_{\frac{\lambda}{\gamma}{f_2^*}}
   (\frac{\lambda}{\gamma}By^{k+1}+\frac{\lambda}{\gamma}v^k).
\end{align}
According to the optimality condition of \eqref{formbasic3Bb02},
\begin{align}\label{inclusion01}
   \frac{\lambda}{\gamma}By^{k+1}+\frac{\lambda}{\gamma}v^k-\frac{\lambda}{\gamma}v^{k+1} \in \frac{\lambda}{\gamma}\partial{f_2^*}(\frac{\lambda}{\gamma}v^{k+1}).
\end{align}
Similarly, according to the optimality condition of \eqref{eqkey3},
\begin{align}\label{inclusion02}
   \frac{\lambda}{\gamma}Bx^* \in \frac{\lambda}{\gamma}\partial{f_2^*}(\frac{\lambda}{\gamma}v^*).
\end{align}
Observing that $\partial f_2^*$ is $\delta$-strongly monotone, we have by  \eqref{inclusion01} and \eqref{inclusion02} that
\begin{align*}
    \langle v^{k+1}-v^*,(B y^{k+1}+v^k-v^{k+1})-Bx^*\rangle \geq \frac{\lambda}{\gamma} \delta \|v^{k+1}-v^*\|^2,
\end{align*}
i.e.,
\begin{align*}
    \langle v^{k+1}-v^*,v^{k+1}-v^k\rangle &\leq  \langle B^T (v^{k+1}-v^*),y^{k+1}-x^*\rangle -\frac{\lambda}{\gamma}\delta \|v^{k+1}-v^*\|^2.
\end{align*}
Thus
\begin{align}\label{eq:v*vk1vk1v*02}
&\|v^{k+1}-v^*\|^2=\|v^{k}-v^*\|^2-\|v^{k+1}-v^k\|^2+2\langle v^{k+1}-v^*, v^{k+1}-v^k\rangle\nonumber\\
\leq &\|v^{k}-v^*\|^2-\|v^{k+1}-v^k\|^2+2 \langle B^T(v^{k+1}-v^*),y^{k+1}-x^*\rangle-\frac{\lambda}{\gamma}\delta \|v^{k+1}-v^*\|^2.
\end{align}

Summing the two inequalities \eqref{eq:v*vk1vk1v*02} and \eqref{eq:xerror}, and then using the same argument for driving  \eqref{midkeyneq}, we arrive at
\begin{align}
(1+\frac{\lambda}{\gamma} \delta)\lambda \|v^{k+1}-v^{*}\|^2+ \|x^{k+1}-x^{*}\|^2
\leq  &
\lambda\|v^k-v^*\|^2+ \|g(x^k)-g(x^*)\|^2 \nonumber
\\
\le& \lambda\|v^k-v^*\|^2+ \eta_2^2\|x^k-x^*\|^2,
\label{keyineq02}
\end{align}
where we have also used the condition \eqref{condition02}.

Let $\eta_1=1/\sqrt{1+\lambda\delta/\gamma}$, $\eta=\max\{\eta_1,\eta_2\}$. It is clear that $0<\eta<1$.  Hence, according to the notation \eqref{notation_lambda_norm}, the estimate \eqref{keyineq02} can be rewritten
as required.
\end{proof}

We note that a linear convergence rate for strongly convex $f_2^*$  and $f_3$ are obtained in \cite{LZ15}.  They  introduced two preconditioned operators for accelerating the algorithm, while a clear relation between the convergence rate and the preconditioned operators is still missing. Meanwhile, introducing preconditioned operators could be beneficial in practice, and we can also introduce a preconditioned operator to deal with $\nabla f_1$ in our scheme. Since the analysis is rather similar to the current one, we will omit it in this paper.

\section{Connections to other algorithms}
\label{seq:Connections}

In this section, we present the connections of the PDFP algorithm to some algorithms proposed previously in the literature. In particular, when $f_3=\chi_C$, due to $\prox_{\gamma f_3}= \mbox{proj}_{C}$, the proposed algorithm \eqref{formbasic3B} is reduced to  PAPA proposed in \cite{KLSX12}
\begin{flalign}
 \label{formbasicC2}
 &\hspace{2em}{(\mbox{PDFP}) \quad }
 \left \{
  \begin{aligned}
   &y^{k+1}=\mbox{proj}_C(x^k-\gamma\nabla {{f_1}}(x^k)-{\lambda} B^T  v^{k}),\\
   &v^{k+1}=(I-\prox_{\frac{\gamma}{\lambda}{{f_2}}})(By^{k+1}+v^k),\\
   &x^{k+1}=\mbox{proj}_C(x^k-\gamma\nabla {{f_1}}(x^k)-{\lambda} B^T  v^{k+1}),\\
 \end{aligned}\right.&
\end{flalign}
where $0<\lambda< 1/\lambda_{\max}(BB^T)$, $0<\gamma<2\beta$. We note that  the conditions of the parameters for the convergence  of PDFP are larger than those in \cite{KLSX12}. Here we still refer  \eqref{formbasicC2} as {PDFP}, since PAPA originally proposed  in \cite{KLSX12} incorporates other techniques such as  diagonal preconditioning.   For the special case $f_3=0$, due to $\prox_{\gamma f_3}=I$, we obtain the PDFP$^2$O scheme proposed in \cite{CHZ13}
\begin{flalign}
 \label{formbasicPDFP2O}
 &\hspace{2em}{(\mbox{PDFP}^2\mbox{O}) \quad }
 \left \{
  \begin{aligned}
   &y^{k+1}=x^k-\gamma\nabla {{f_1}}(x^k)-{\lambda} B^T  v^{k},\\
   &v^{k+1}=(I-\prox_{\frac{\gamma}{\lambda}{{f_2}}})(By^{k+1}+v^k),\\
   &x^{k+1}=x^k-\gamma\nabla {{f_1}}(x^k)-{\lambda} B^T  v^{k+1},\\
 \end{aligned}\right.&
\end{flalign}
where $0<\lambda\leq 1/\lambda_{\max}(BB^T)$, $0<\gamma<2\beta$.
Based on PDFP$^2$O, we also proposed PDFP$^2$O$_C$ in \cite{CHZ1302} for $f_3=\chi_C$ as
\begin{flalign}
 \label{formbasicC1}
 &\hspace{2em} {(\mbox{PDFP}^2\mbox{O}_{C})\quad}
 \left \{
  \begin{aligned}
   &y^{k+1}=x^k-\gamma\nabla {{f_1}}(x^k)-{\lambda} B^T  v_1^{k}-{\lambda}v_2^{k},\\
   &v_1^{k+1}=(I-\prox_{\frac{\gamma}{\lambda}{{f_2}}})(By^{k+1}+v_1^k), \\
   &v_2^{k+1}=(I-\mbox{proj}_{C})(y^{k+1}+v_2^k), \\
   &x^{k+1}=x^k-\gamma\nabla {{f_1}}(x^k)-{\lambda} B^T  v_1^{k+1}-{\lambda}v_2^{k+1},\\
 \end{aligned}\right.&
\end{flalign}
where $0<\lambda\leq 1/(\lambda_{\max}(BB^T)+1)$, $0<\gamma<2\beta$.
Similar technique of extension to multi composite functions have also been used in \cite{C13,LSXZ15,TZW15}.  Compared to PDFP \eqref{formbasicC2}, the algorithm PDFP$^2$O$_C$  introduces an extra variable, while PDFP requires two times projections.
Most importantly, the primal variable at each iterate of PDFP is feasible, but maybe not for that of PDFP$^2$O$_C$. In addition, the permitted ranges of the parameters are also tighter in PDFP$^2$O$_C$.



The other most related algorithm to the PDFP algorithm \eqref{formbasic3B} is  the  algorithm proposed by L. Condat in \cite{C13}. For the special case with $f_1=0$, Condat's algorithm reduces to PDHG method  in \cite{CP11}.  By grouping multi-block as a single block, the authors in \cite{TZW15}  extended the PDHG algorithm  \cite{PC11} to multi composite functions penalized problems. The authors in \cite{LSXZ15}  proposed a class of multi-step fixed-point proximity
algorithms,  including several existing algorithms as special examples, for example the algorithms in \cite{CP11,PC11}.
The three-block method proposed by Davis and Yin in \cite{DY15} are based on operator splitting  but subproblem solving is required when it is applied to  solve \eqref{PDFP2O3B:eqbasic} for the  case $B\neq I$.
Li and Zhang \cite{LZ15} also introduce preconditioned operators based on the techniques present in \cite{LSXZ15} and including  Condat's algorithm in \cite{C13} as a special case, and  further introduce quasi-Newton and the overrelaxation strategies to accelerate the algorithms. Specifically, we compare the PDFP algorithm \eqref{formbasic3B} with the basic Algorithm 3.2 in \cite{C13}.

In the following, we mainly compare PDFP to Condat's algorithm  \cite{C13} for  a simple presentation. We first change the form of PDFP algorithm \eqref{formbasic3B}  by using Moreau's identity, see \eqref{lem_Moreau_decomposition}, i.e.
\[
   (I-\prox_{\frac{\gamma}{\lambda}{f_2}})(By^{k+1}+ {v}^k)
   =\frac{\gamma}{\lambda}\prox_{\frac{\lambda}{\gamma}{f_2^*}}
   (\frac{\lambda}{\gamma}By^{k+1}+\overline{v}^k),
\]
 where $\overline{v}^k= \DF{\lambda}{\gamma}v^k$. A direct comparison is presented in Table \ref{table:cmp_Condat_PDFP2O}. From Table \ref{table:cmp_Condat_PDFP2O}, we can see that the ranges of the parameters in Condat's algorithm are relatively smaller than {PDFP}. Also since the condition for Condat's algorithm is mixed with all the parameters, it is not always easy to choose them in practice. This is also pointed out in \cite{C13}. While the rules for the parameters in {PDFP} are separate, and they can be chosen independently according to the Lipschitiz constant and the operator norm of $BB^T$. In this sense, our parameter rules are relatively more practical. In the numerical experiments, we can  set $\lambda$ to be close to $1/\lambda_{\max}(BB^T)$ and  $\gamma$ to be close to $2\beta$ for most of tests. Nevertheless, PDFP has an extra step \eqref{formbasic3Ba} compared to Condat's algorithm and the computation cost may increase  due to the computation of $\prox_{{\gamma}{f_3}}$. In practice, this step is often related to $\ell_1$ shrinkage, so the cost could be still ignorable in practice.

\begin{table}[!htp] \small
\begin{center}
 \caption{The comparison  between Condat ($\rho_k=1$) and  {PDFP}.}
 \label{table:cmp_Condat_PDFP2O}
 \begin{tabular}{|p{1.2cm}|p{7.3cm}|p{6.2cm}|}
   \hline
   & \centering{Condat ($\rho_k=1$)} &  \quad\quad\quad\quad\quad\quad\quad  {PDFP}\\
   \hline
   \multirow{3}{*}{\small   Form}&
   & \small    $y^{k+1}=\prox_{{\gamma}{f_3}}(x^k-\gamma\nabla {f_1}(x^k)-{\gamma} B^T \overline{ v}^k)$\\
   &\small   $\overline{v}^{k+1}=\prox_{\sigma{f_2^*}}(\sigma Bx^{k}+\overline{v}^k)$
   &\small   $\overline{v}^{k+1}=\prox_{\frac{\lambda}{\gamma}{f_2^*}}(\frac{\lambda}{\gamma}By^{k+1}+\overline{v}^k)$\\
   &\small   $x^{k+1}=\prox_{{\tau}{f_3}}(x^k-\tau\nabla {f_1}(x^k)-{\tau} B^T  (2\overline{v}^{k+1}-\overline{v}^{k}))$
   &\small   $x^{k+1}=\prox_{{\gamma}{f_3}}(x^k-\gamma\nabla {f_1}(x^k)-{\gamma} B^T  \overline{v}^{k+1})$\\
   \hline
   \small $  f_1\neq 0$&\small   $ \sigma\tau  \lambda_{\max}(BB^T)+\tau/(2\beta)\leq 1$ & \small   $ 0<\lambda< 1/\lambda_{\max}(BB^T)$, $0<\gamma<2\beta$   \\
   \hline
   \small $  f_1=0$&\small   $0<\sigma\tau\leq1/\lambda_{\max}(BB^T)$ &\small   $0<\lambda< 1/\lambda_{\max}(BB^T)$, $0<\gamma<+\infty$  \\
   \hline
   \small   Relation& \multicolumn{2}{|c|}{\small   $\sigma=\lambda/\gamma$, $\tau=\gamma$}\\
   \hline
 \end{tabular}
 \end{center}
\end{table}

\section{Numerical experiments}
\label{sec:Numerical_experiments}

In this section, we will apply the  {PDFP} algorithm to { solve}  two problems:   the fused LASSO penalized problems and parallel Magnetic Resonance Imaging (pMRI) reconstruction.  All the experiments are implemented under MATLAB7.00 (R14) and conducted on a computer with Intel (R) core (TM) i5-4300U CPU@1.90G.

\subsection{The fused LASSO penalized problems}\label{subsec:FLASSO}

The fused LASSO (Least absolute shrinkage and selection operator) penalized problems is proposed for group variable selection, and one can refer to \cite{LJY11,LYY10} for more details for the applications of this model. It  can be described as
\begin{align*}
\underset{x\in {\mathbb{R}^n}}{\mbox{ min}}\quad \frac 1 2\|Ax-a\|^2+ \mu_1 \sum_{i=1}^{n-1}|x_{i+1}-x_i|+\mu_2 \|x\|_1. 
\end{align*}
Here $A\in \mathbb{R}^{r\times n}$, $a\in \mathbb{R}^r$. The row of $A$: $A_i$ for $i=1,2,\cdots, r$ represent the  $i\_$th observation of the independent variables and $a_i$ denotes the response variable, and the vector  $x\in \mathbb{R}^{n}$ is the regression coefficient to recover. The first term is corresponding to the data-fidelity term, and the last two terms aim to ensure the sparsity in
both $x$ and their successive differences in $x$.
Let
$
   B=\begin{pmatrix}
       -1&1\\
       &-1&1\\
       &&\ddots&\ddots\\
       &&&-1&1
     \end{pmatrix}.
$ Then the forgoing problem can be reformulated as 
\begin{align}
 \underset{x\in {\mathbb{R}^n}}{\mbox{min}}\quad \frac 1 2\|Ax-a\|^2+\mu_1 \|Bx\|_1+ \mu_2 \|x\|_1.\label{fused_LASSO02}
\end{align}
For this example, we can set ${f_1}(x)=\frac 1 2\|Ax-a\|^2$, ${f_2}  =\mu_1\|\cdot\|_1$,   ${f_3} =\mu_2\|\cdot\|_1$.
We want to show that the {PDFP} algorithm \eqref{formbasic3B} can be applied to solve this  generic class of problem \eqref{fused_LASSO02} directly and easily.

The following tests are designed for the simulation. We set $r=500$, $n=10000$, and the data $a$ is generated as $Ax+\sigma e$, where $A$ and $e$ are random matrices whose elements are normally distributed with zero mean and variance 1, and  $\sigma=0.01$,  and $x$ is a generated sparse vector, whose nonzeros elements are showed in Figure \ref{figure:FLASSO_cmp}  by green '+'.  We set $\mu_1=200$, $\mu_2=20$ and the maximum iteration number as $Itn=1500$.

We compare the PDFP algorithm with Condat's algorithm \cite{C13}. For the PDFP algorithm, the parameter $\lambda$ and $\gamma$ are chosen according to Theorem \ref{theorem:convergence}. In practice, we set $\lambda$ to be close to $1/\lambda_{\max}(BB^T)$ and  $\gamma$ to be close to $2\beta$. Here we set $\lambda=1/4$ as the $n-1$ eigenvalues of  $BB^T$ can be analytically computed as $2-2 \cos(i\pi /n), i = 1, 2,\cdots, n-1$  and  $\gamma=1.99/\lambda_{\max}(A^TA)$.  For Condat's algorithm, we set $\lambda= 0.19/4$, $\gamma=1.9/\lambda_{\max}(A^TA)$, which is chosen for a relative better numerical performance. The computation time, the attained objective function values, and the relative errors to the true solution are close for Condat's algorithm and {PDFP}.  From Figure \ref{figure:FLASSO_cmp},  we see that both Condat's algorithm and {PDFP} can quite correctly recover the positions of the non-zeros and the values.

\begin{figure}[!htp]\centering\small
\caption{Recovery results for fused LASSO with Condat's algorithm and {PDFP} . }
 \label{figure:FLASSO_cmp}
 \begin{tabular}{cc}
   \includegraphics[width=0.45\textwidth]{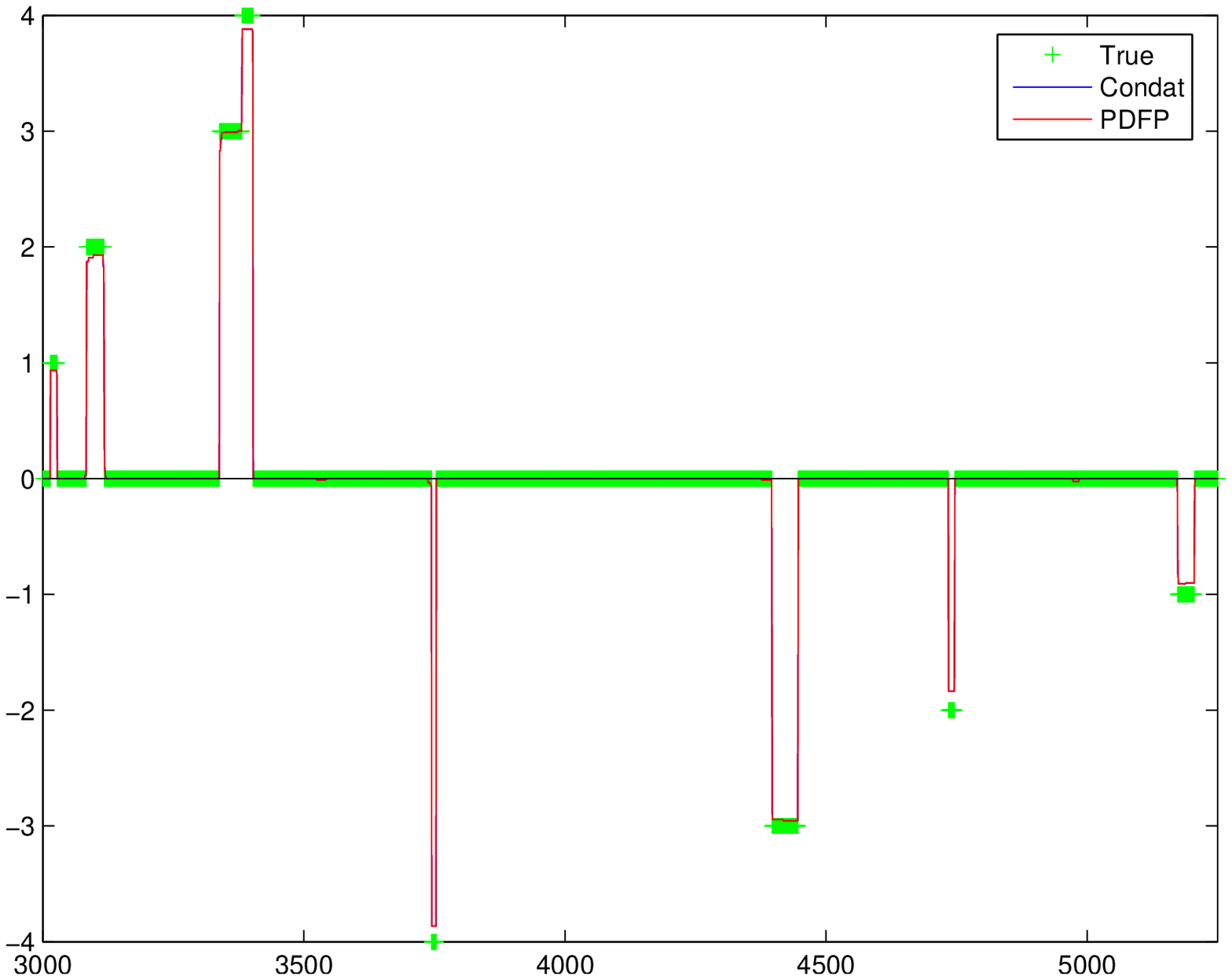}&
   \includegraphics[width=0.45\textwidth]{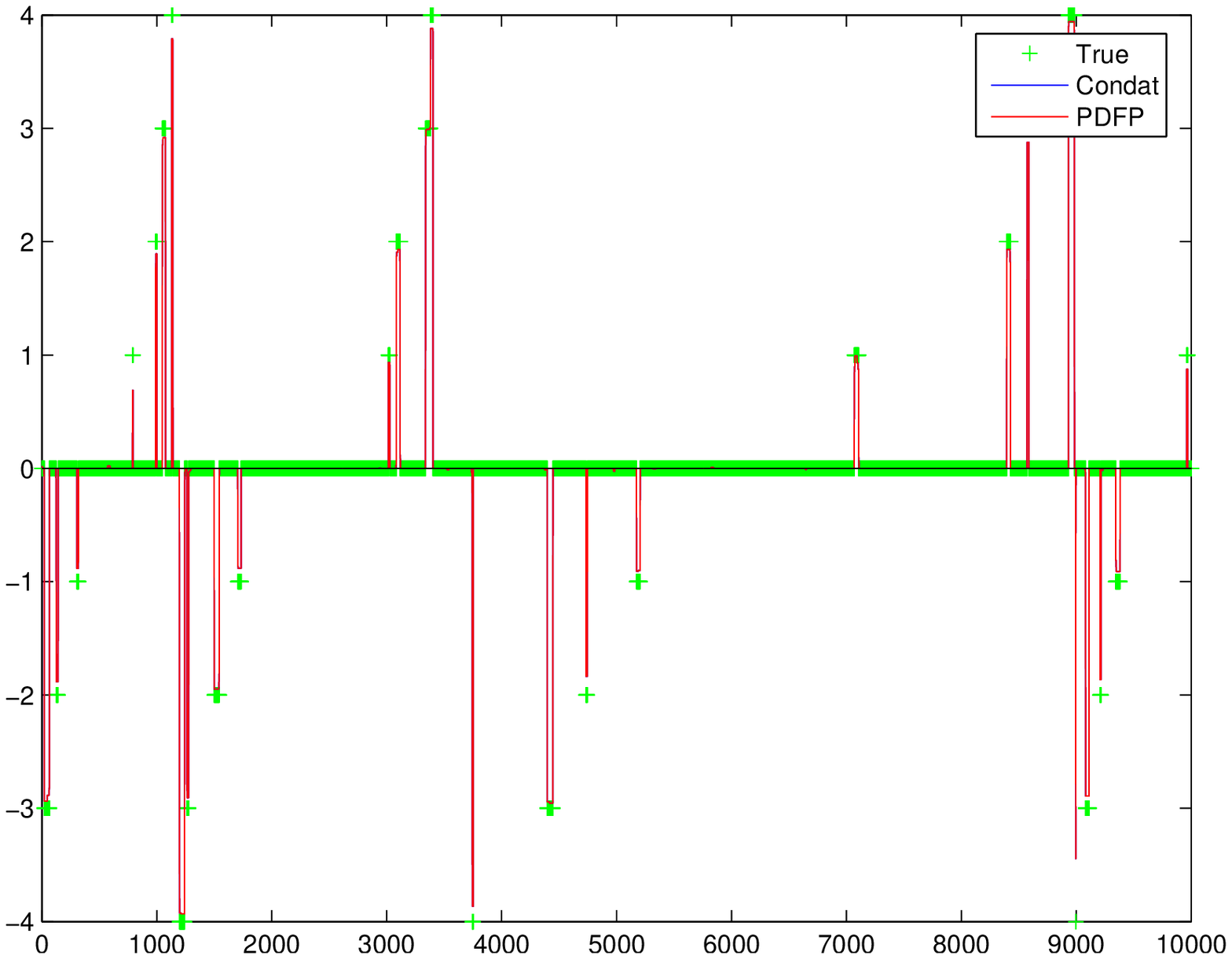}
   \end{tabular}
\end{figure}

\subsection{Image restoration with nonnegative constraint and sparse regularization}
\label{subsec:pMRI}

A general image restoration problem with nonnegative constraint and sparse regularization can be written as
\begin{align}
  \underset{x\in {C}}{\mbox{min}}\quad\frac 1 2\|Ax-a\|^2 +\mu \|Bx\|_1,  \label{eq_TVL2}
\end{align}
 where $A$ is some bounded linear operator describing the image formation process, $\|Bx\|_{1}$ is  the usual $\ell_1$ based regularization  in order to promote sparsity under the transform $B$, $\mu>0$ is the regularization parameter. Here we use isotropic  total variation as the regularization functional, thus the matrix $B$ represents for the discrete gradient operator. For this example, we can set $f_1(x)=\frac 1 2\|Ax-a\|^2$, $f_2=\mu \|\cdot\|_1$, and ${f_3}=\chi_C$.

We consider pMRI reconstruction,  where $A=(A_1^T,A_2^T,\cdots, A_N^T)^T$ for each $A_j$ is composed of a  diagonal
downsampling operator $D$, Fourier transform $F$ and a  diagonal coil sensitivity mapping $S_j$ for receiver $j$, i.e. $A_j=DFS_j$ and $S_j$ are often estimated in advance. It is well known in total variation application that $\lambda_{\max}(BB^T)=8$.  The related Lipschitz constant of $ \nabla f_1$ can be estimated as $\beta=1$. Therefore the two parameters in PDFP are set as $\lambda=1/8$ and $\gamma=2$. The same simulation setting as in  \cite{CHZ13} is used in this experiment and we still use artifact power (AP) and two-region signal to noise ratio (SNR) to measure image quality. One may refer to \cite{CHZ13,JSR07} for more details.

In the following, we compare PDFP algorithm with the previous proposed algorithms  PDFP$^2$O \eqref{formbasicPDFP2O} and PDFP$^2$O$_{C}$ \eqref{formbasicC1}. From Figure \ref{figure:Spine_pMRI_cmp} and \ref{figure:Brain_pMRI_cmp}, we can first see that the introduction of nonnegative constraint in the model \eqref{eq_TVL2} is beneficial and we can recover a better solution with higher two-region SNR and lower AP value. The nonnegative constraint leads to a faster convergence for a stable recovery. Secondly, PDFP$^2$O$_{C}$ and {PDFP}  are both efficient. For a subsampling rate $R=2$, PDFP$^2$O$_{C}$ and {PDFP} can both recover better solutions in terms of  AP values compared to PDFP$^2$O under the same iterative numbers.
For $R=4$, the solutions of PDFP$^2$O$_{C}$ and {PDFP} have better AP values than those of PDFP$^2$O, but only use half iteration numbers of PDFP$^2$O. The computation time for PDFP is slightly less than PDFP$^2$O$_{C}$. Finally, the iterative solutions of PDFP are always feasible, which could be useful in practice.

\begin{figure}[!htp]\centering\small
\caption{Recovery results from  four-channel in-vivo spine data with the subsampling ratio $R=2, 4$.  For PDFP$^2$O and {PDFP}, $\lambda=1/8$, $\gamma=2$ and for PDFP$^2$O$_{C}$, $\lambda=1/9$, $\gamma=2$. }
 \label{figure:Spine_pMRI_cmp}
 \begin{tabular}{cccccc}
   &\small PDFP$^2$O &\small PDFP$^2$O$_{C}$ &\small {PDFP} \\
   R=2&
   \includegraphics[width=0.27\textwidth]{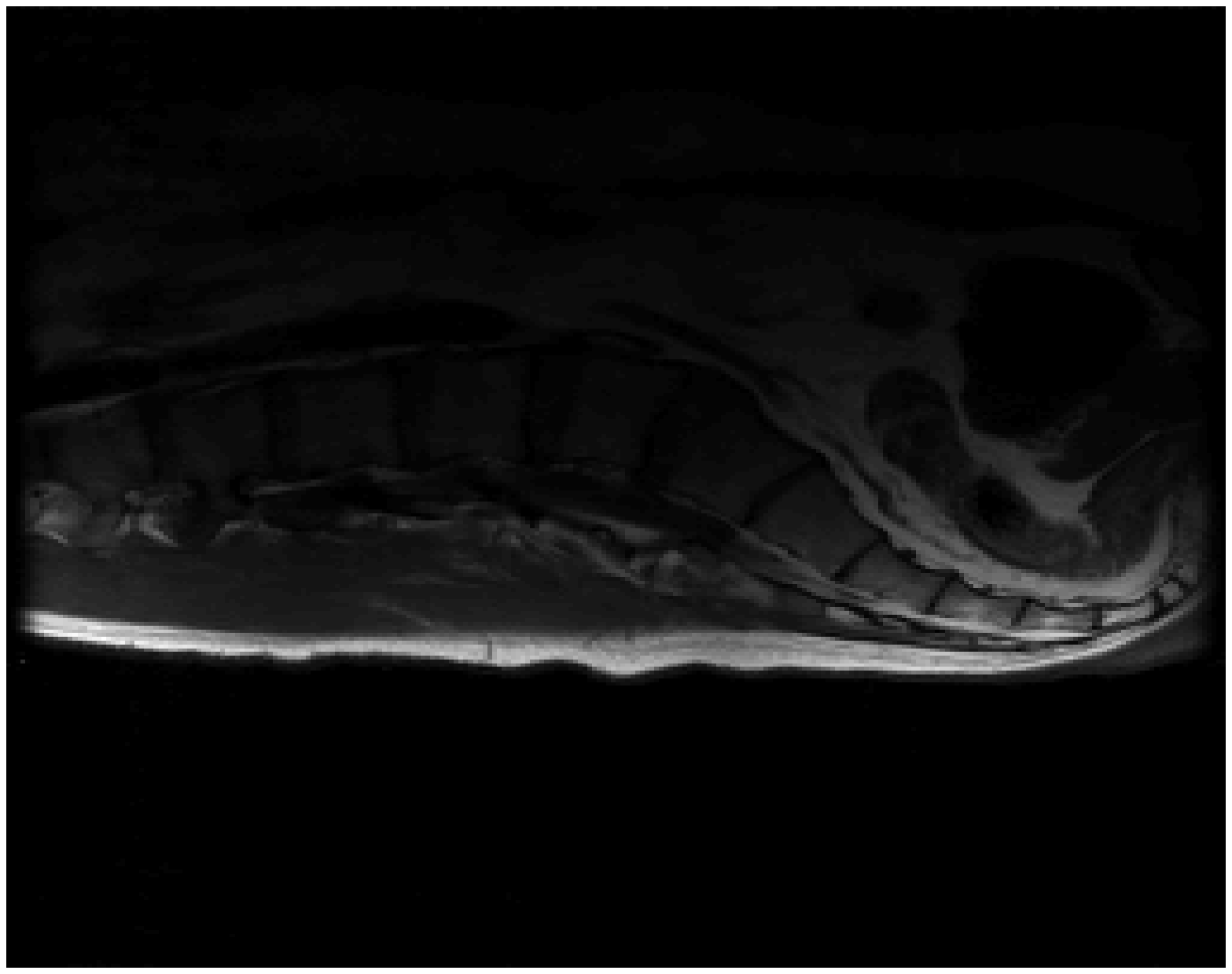}&
   \includegraphics[width=0.27\textwidth]{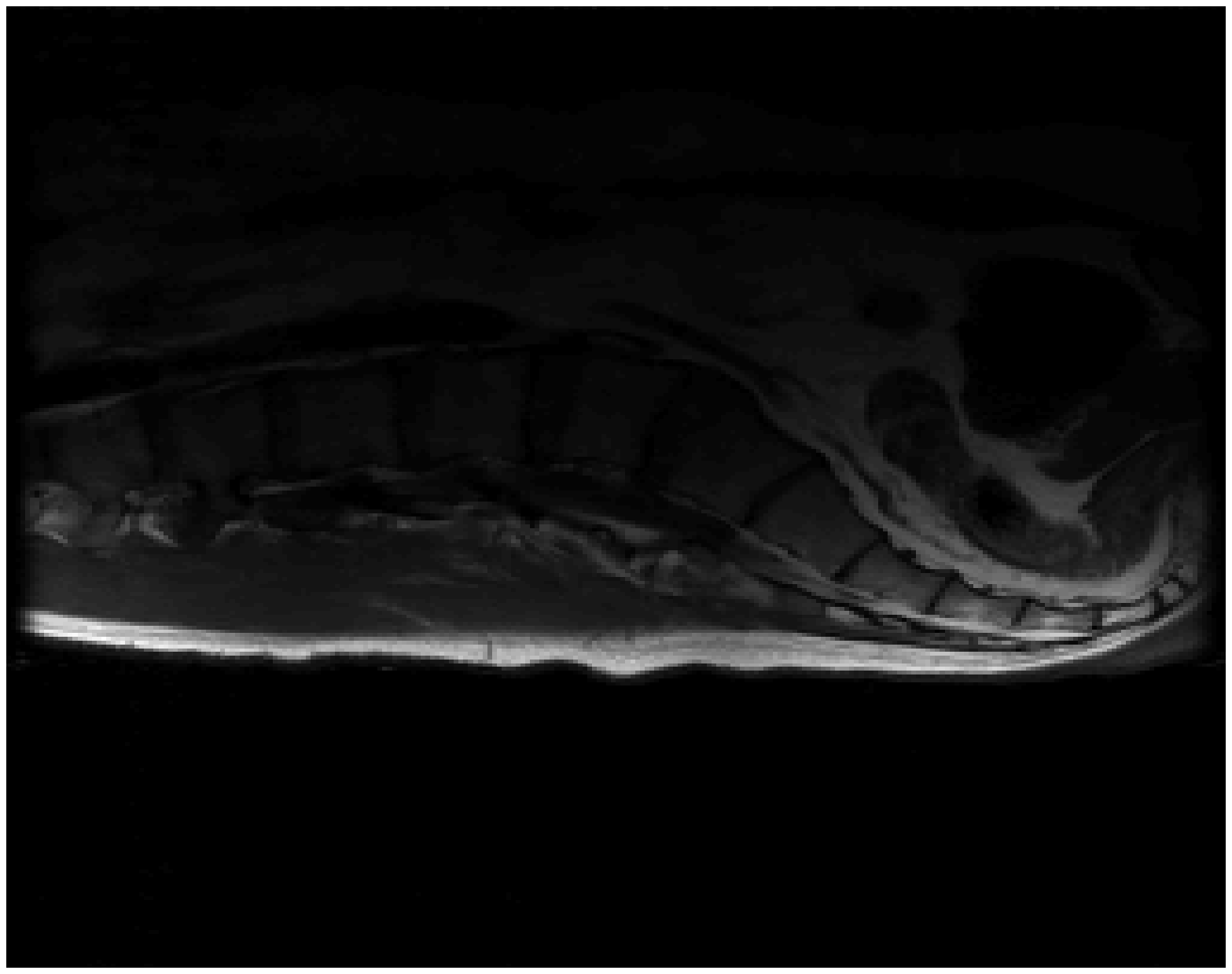}&
   \includegraphics[width=0.27\textwidth]{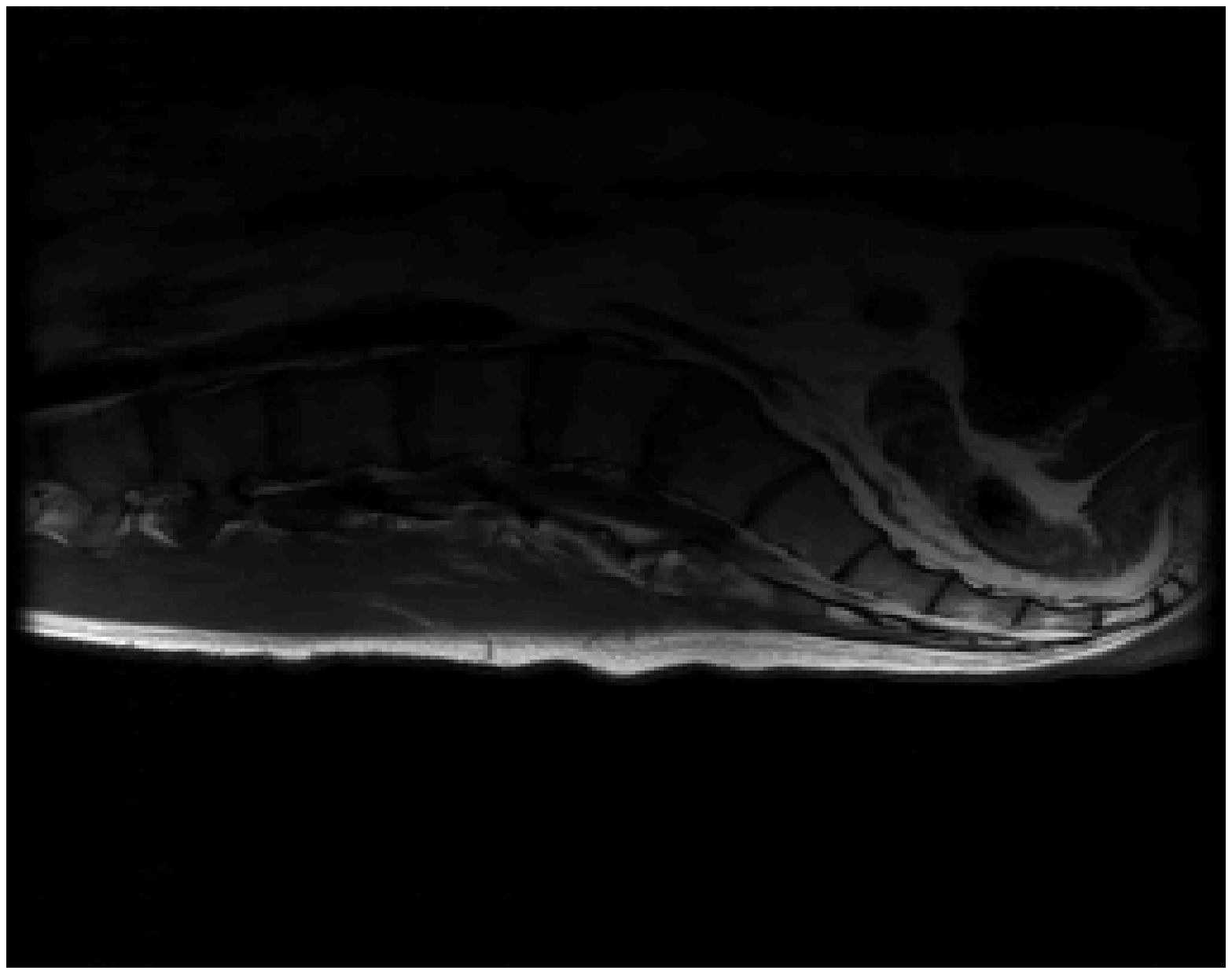}&
   \\
   AP  &  0.002523  &  0.001294  &  0.001021  \\
   SNR  & 34.94  & 35.35  & 36.01  \\
   $Itn$  & 8  & 8  & 8   \\
   time  &   0.73  &   0.75  &   0.67  \\
   \\
   R=4&
   \includegraphics[width=0.27\textwidth]{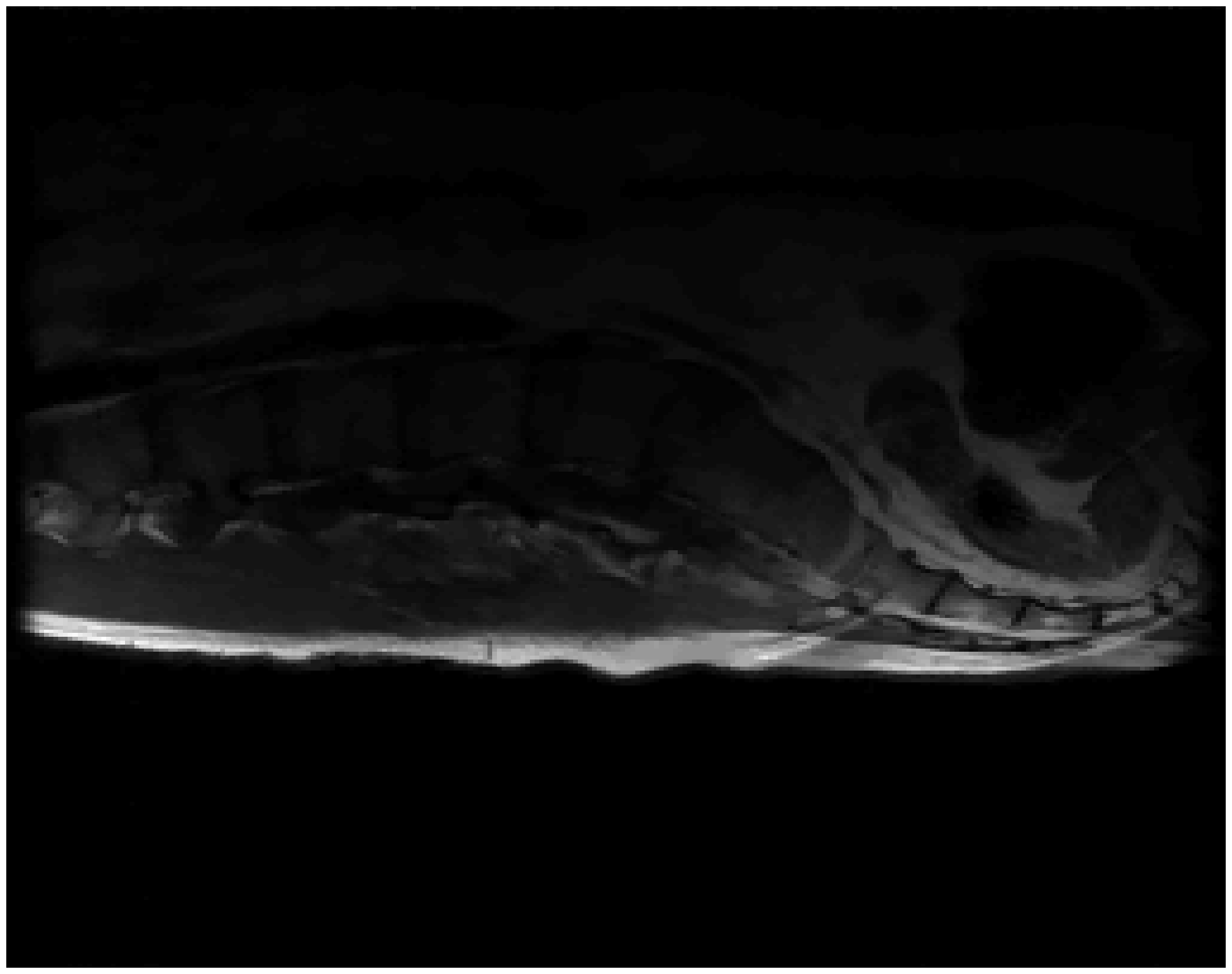}&
   \includegraphics[width=0.27\textwidth]{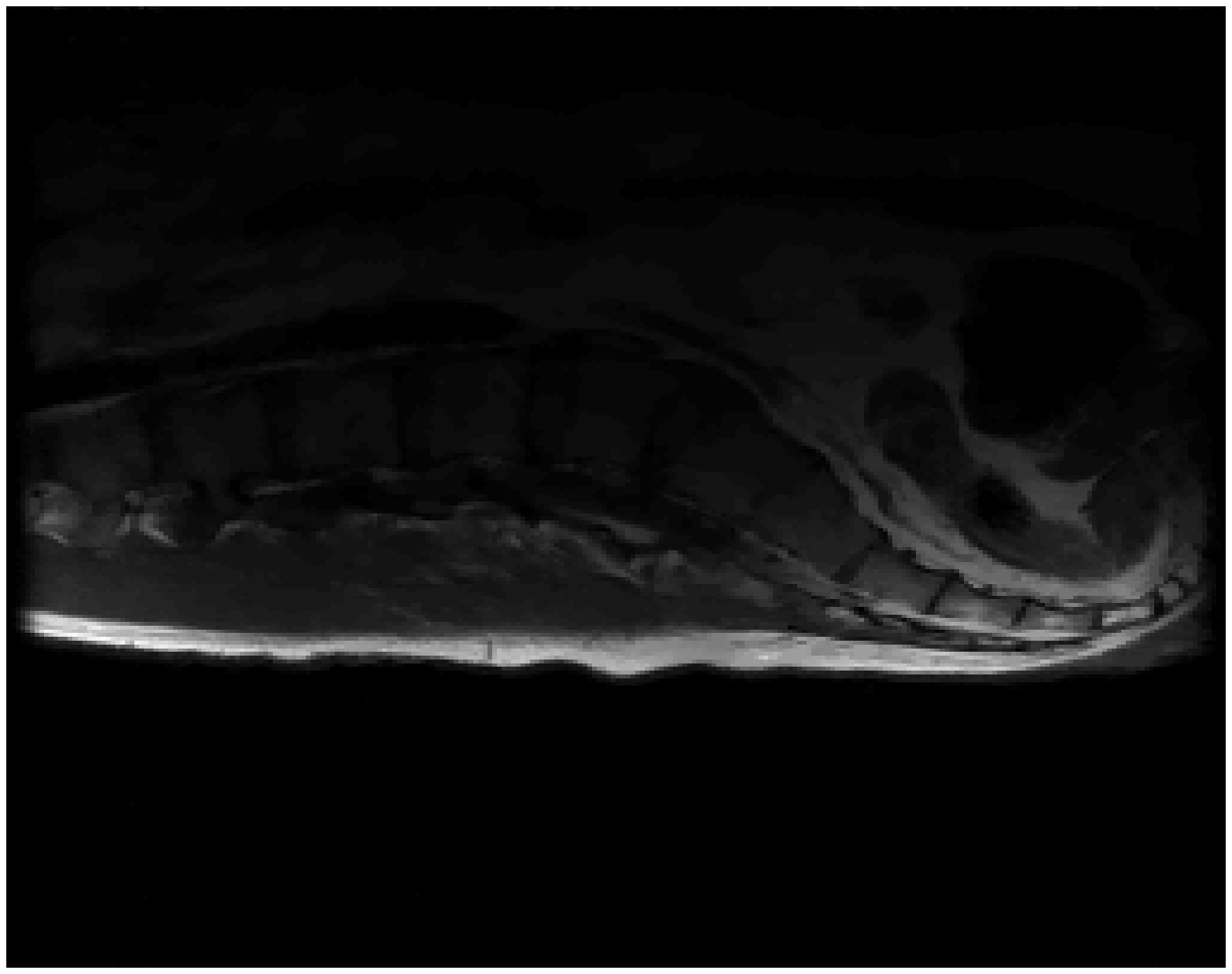}&
   \includegraphics[width=0.27\textwidth]{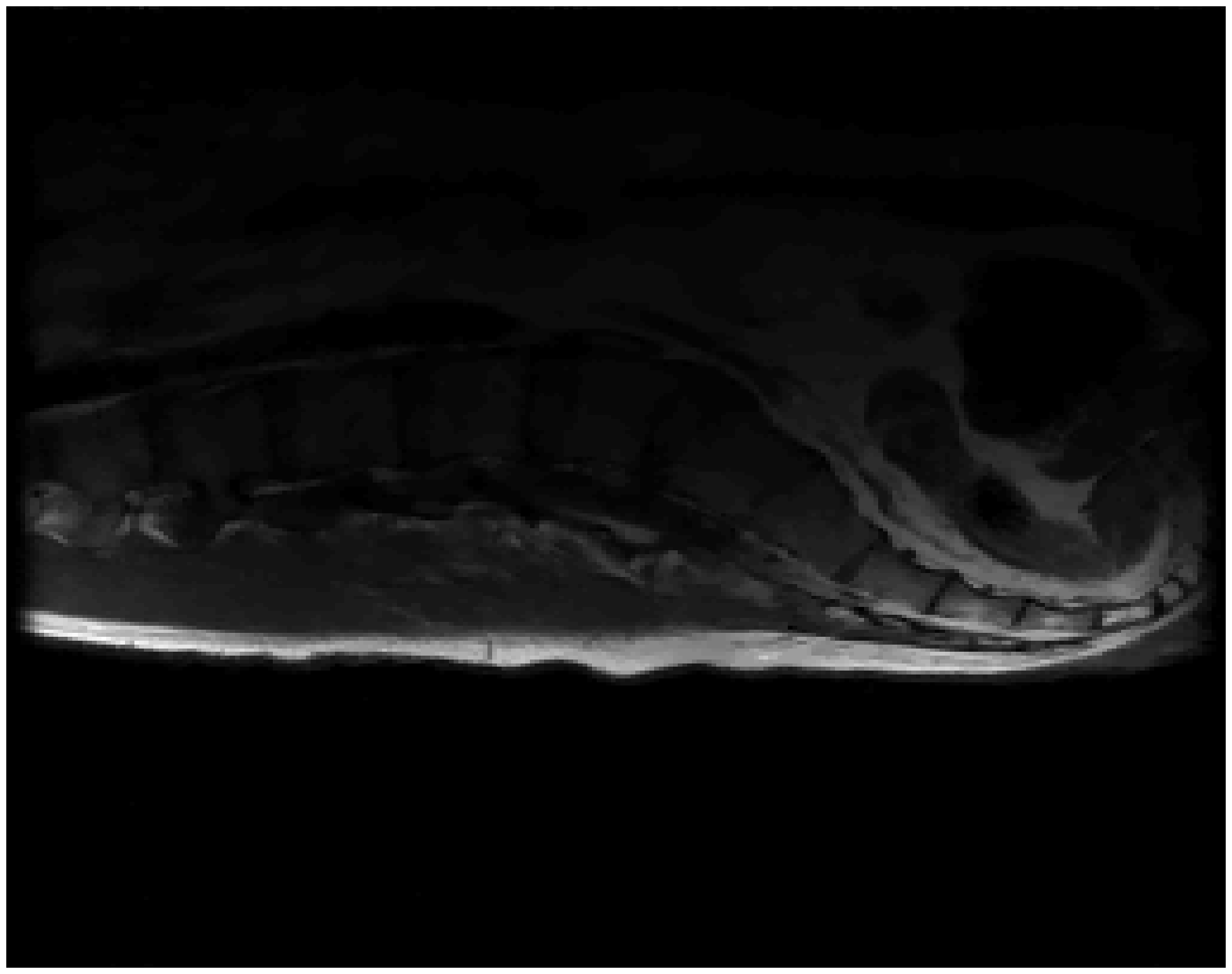}&
   \\
   AP  &  0.040011  &  0.009718  &  0.009802\\
   SNR  & 38.06  & 39.55  & 39.57\\
   $Itn$  & 500  & 250  & 250\\
   time  &  43.76  &  22.35  &  19.30  \\
   \end{tabular}
\end{figure}

\begin{figure}[!htp]\centering\small
\caption{Recovery results from eight-channel in-vivo brain data with the subsampling ratio $R=2,4$. For PDFP$^2$O and {PDFP}, $\lambda=1/8$, $\gamma=2$ and for PDFP$^2$O$_{C}$, $\lambda=1/9$, $\gamma=2$. }
 \label{figure:Brain_pMRI_cmp}
 \begin{tabular}{cccccc}
   &\small PDFP$^2$O &\small PDFP$^2$O$_{C}$ &\small {PDFP} \\
   R=2&
   \includegraphics[width=0.27\textwidth]{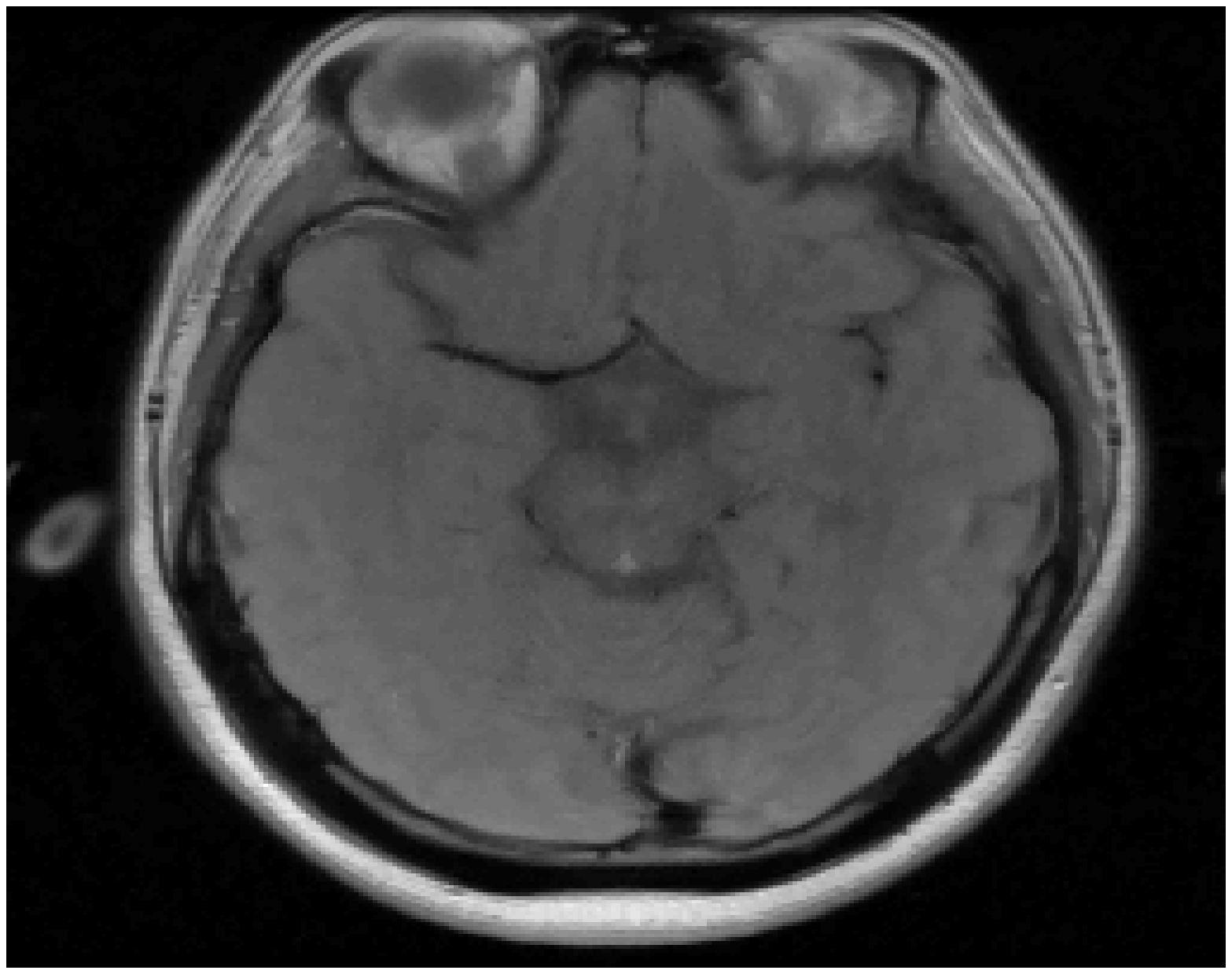}&
   \includegraphics[width=0.27\textwidth]{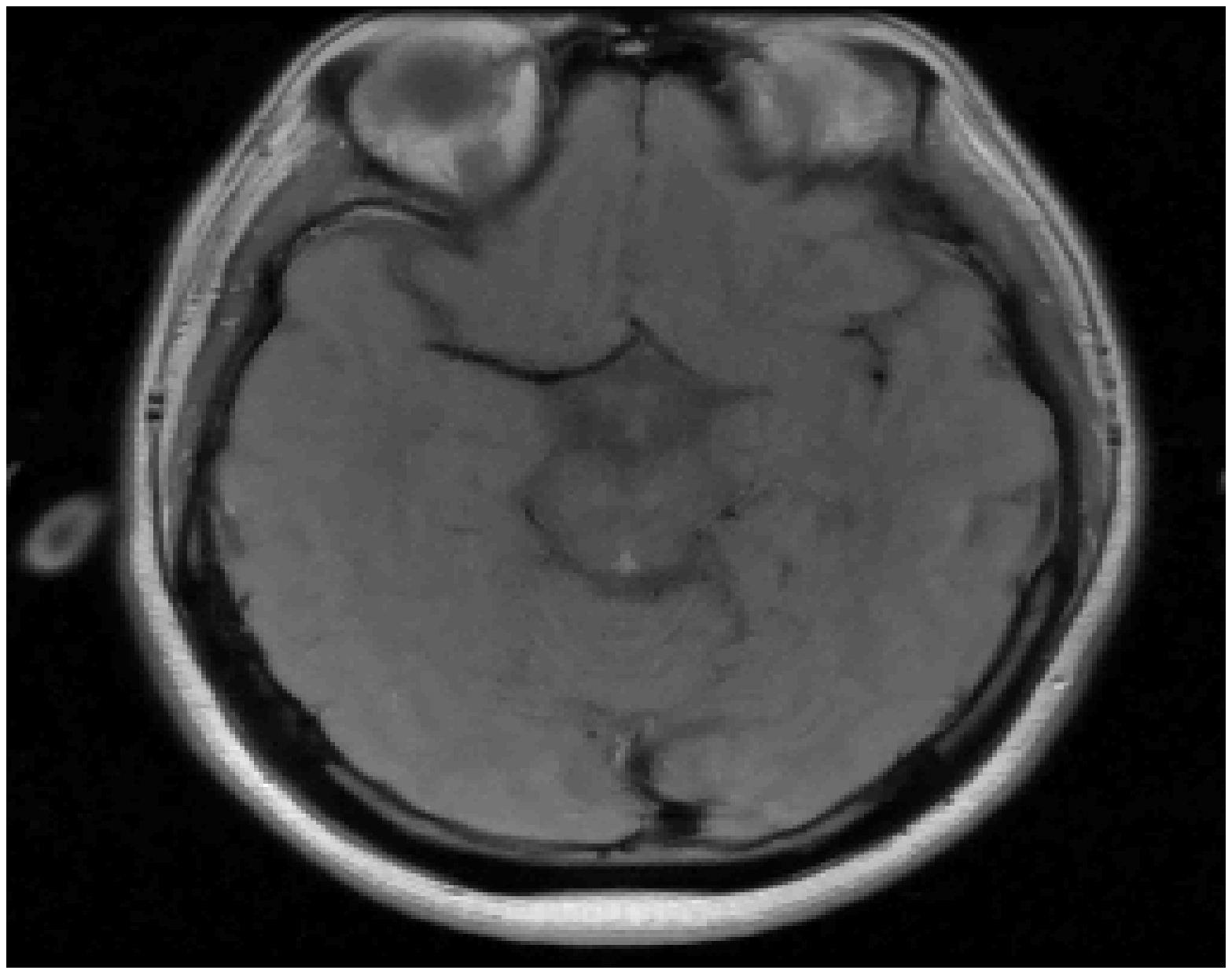}&
   \includegraphics[width=0.27\textwidth]{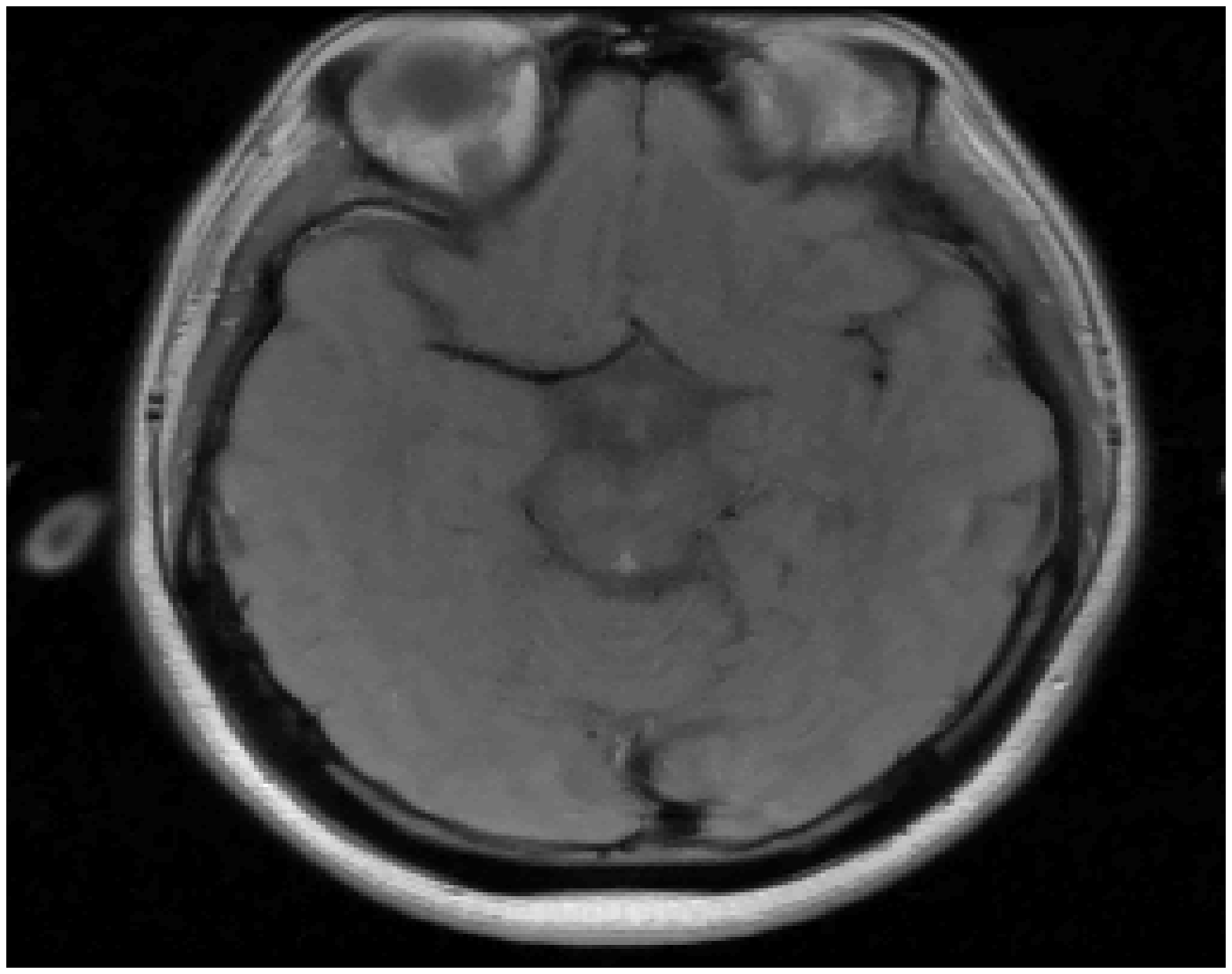}&
   \\
   AP  &  0.000822  &  0.000469  &  0.000465 \\
   SNR  & 39.36  & 39.75  & 40.38  \\
   $Itn$  & 25  & 25  & 25\\
   time  &   3.96  &   4.01  &   3.63  \\
   \\
   R=4&
   \includegraphics[width=0.27\textwidth]{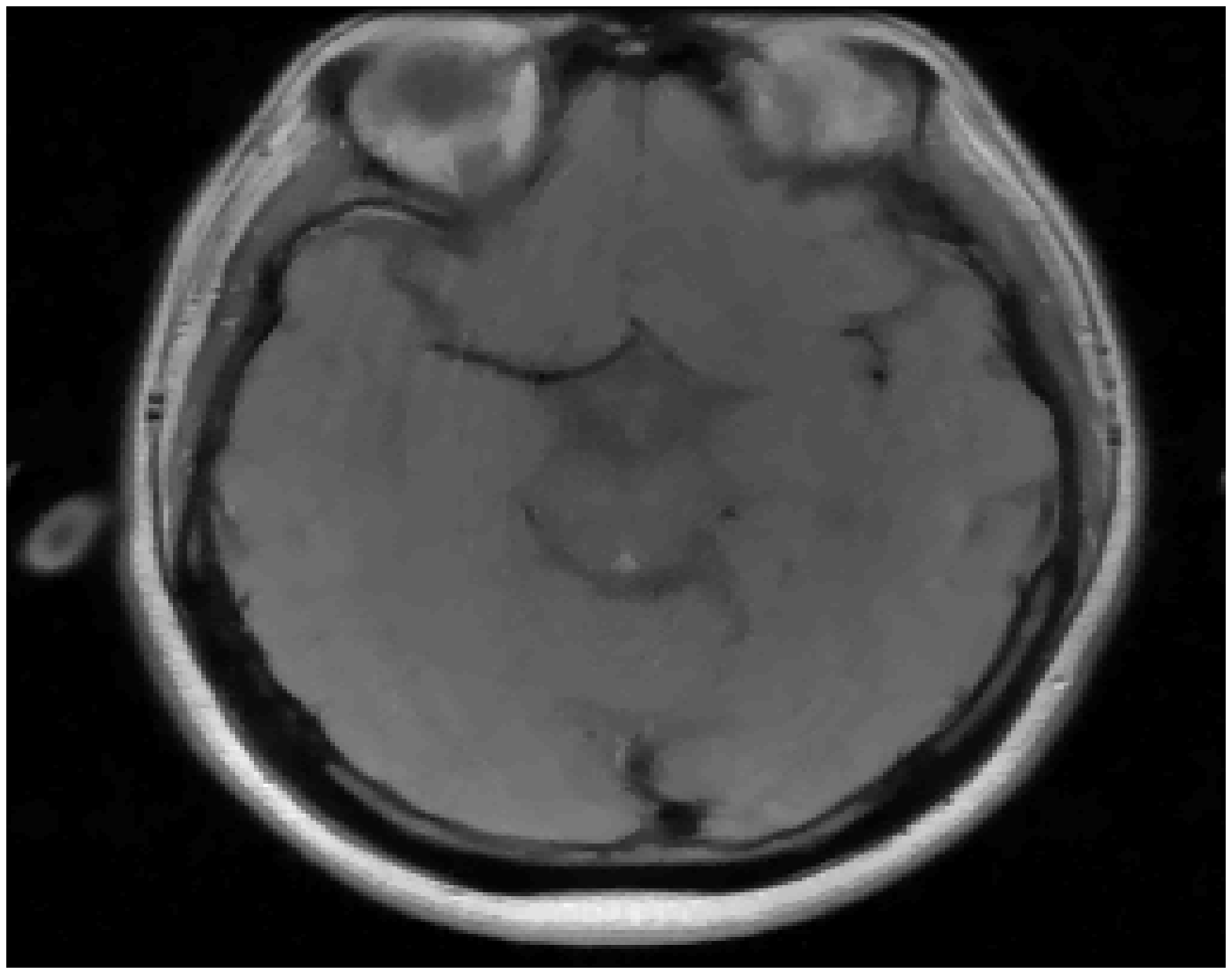}&
   \includegraphics[width=0.27\textwidth]{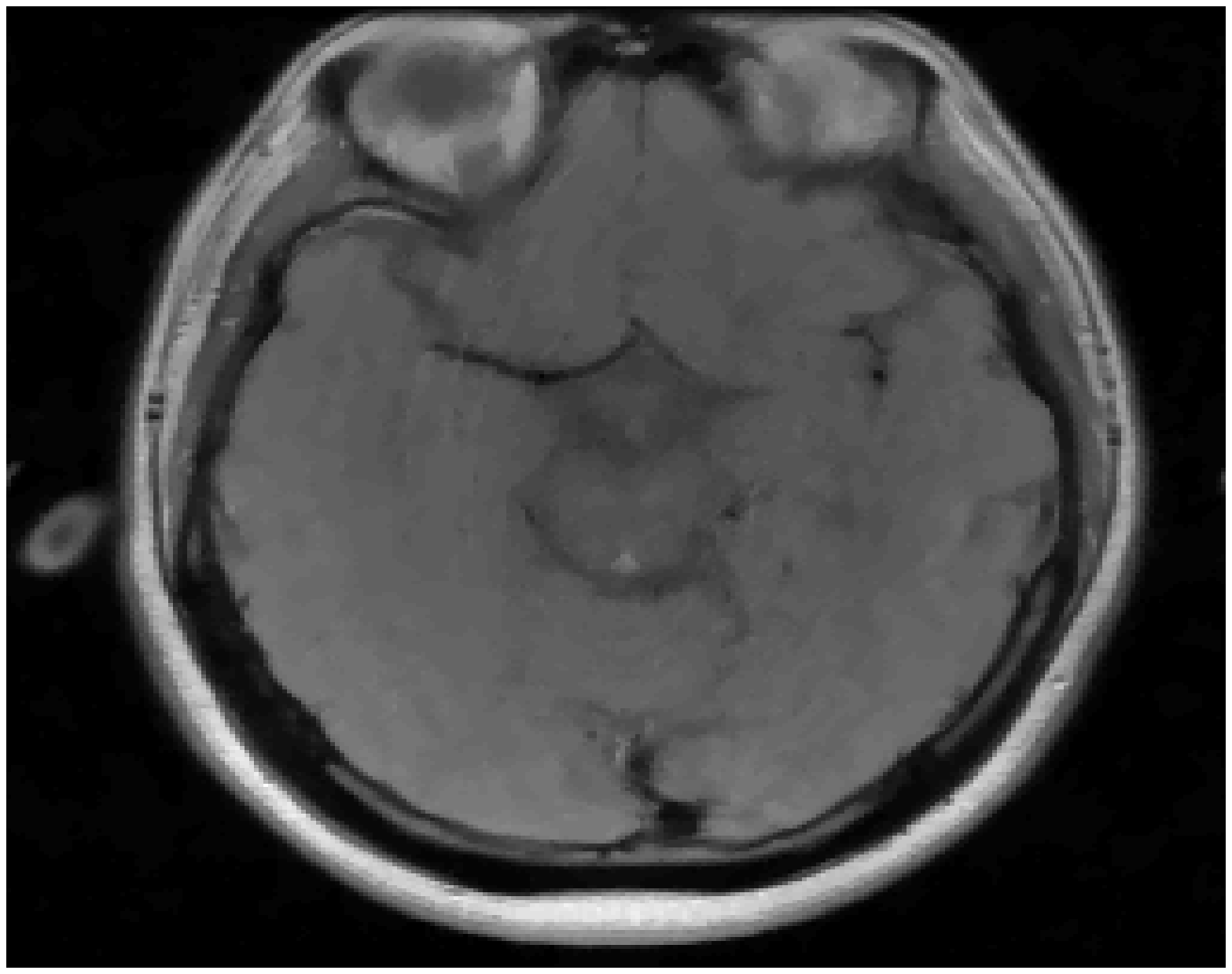}&
   \includegraphics[width=0.27\textwidth]{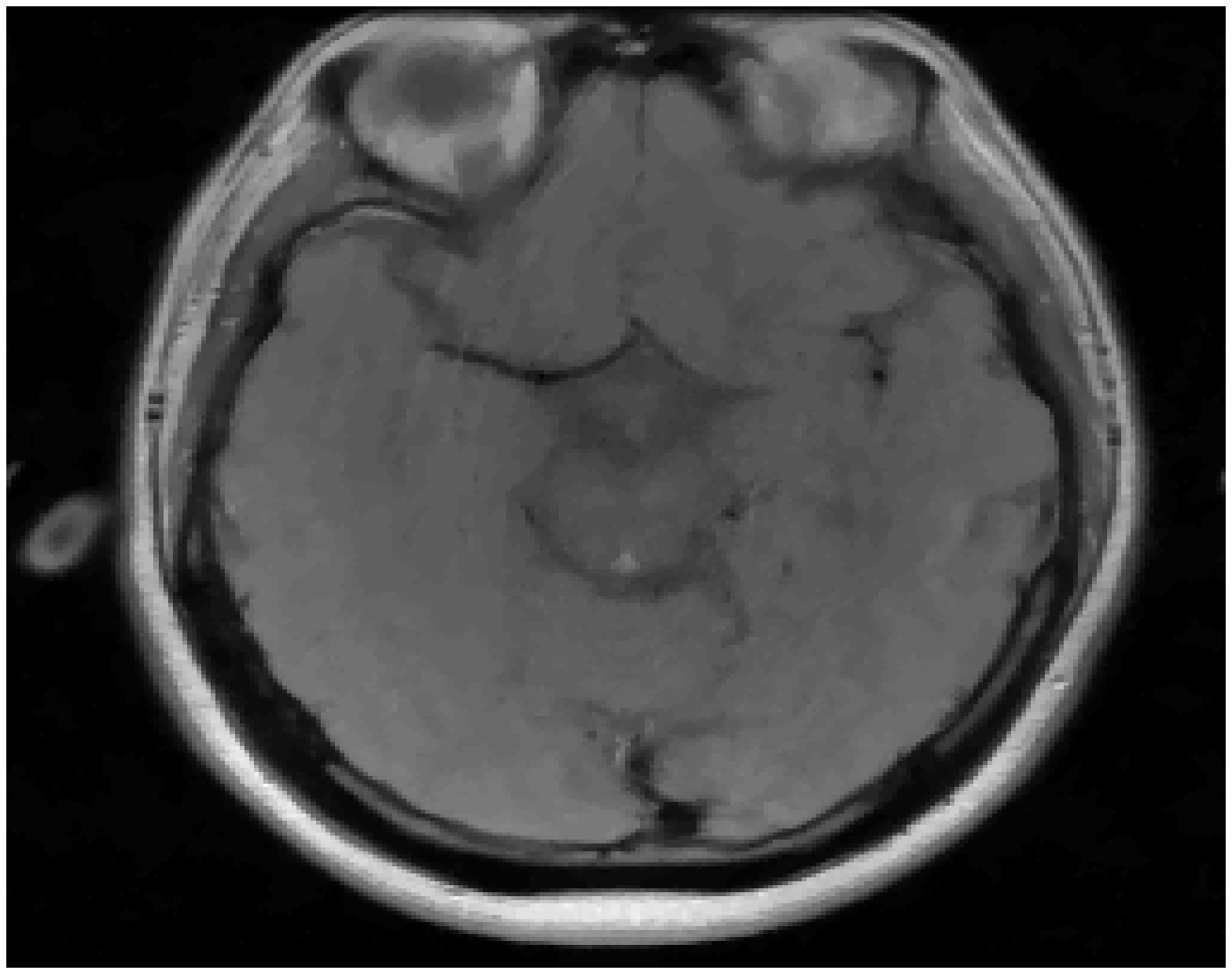}&
   \\
   AP  &  0.002502  &  0.001528  &  0.001535  \\
   SNR  & 43.06  & 43.86  & 44.13  \\
   $Itn$  & 150  & 75  & 75  \\
   time  &  23.02  &  11.74  &  10.79  \\
   \end{tabular}
\end{figure}

\section{Conclusion}
 We have extended the algorithm PAPA \cite{KLSX12} and PDFP$^2$O \cite{CHZ13} to derive a  primal-dual fixed-point algorithm {PDFP} (see \eqref{formbasic3B}) for solving the minimization problem of three-block convex separable functions \eqref{PDFP2O3B:eqbasic}. The proposed
 {PDFP} algorithm  is  a symmetric and fully splitting scheme, only involving explicit gradient and  linear operators  without any
inversion and subproblem solving, when the proximity
operator of nonsmooth functions can be  easily handled.  The scheme can be easily adapted to many inverse problems involving many terms minimization and it is suitable for large scale parallel implementation. In addition, the parameter range determined by the convergence analysis is rather simple and clear, and it could be useful for practical application. Finally as discussed in Section 5 in  \cite{C13}, we can also extend  the current {PDFP} algorithm to solve  multi-block composite  (more than three)  minimization problems. Preconditioning operators, as proposed in  \cite{TZW15,LZ15} can be also introduced to accelerate {PDFP}, which could be a  future work for some specific applications.

\section*{Acknowledgement}
P. Chen was partially supported by the PhD research startup foundation of Taiyuan University of Science and Technology (No. 20132024). J. Huang was partially supported by NSFC (No. 11171219). X. Zhang was  partially supported by NSFC (No. 91330102 and GZ1025) and 973 program (No. 2015CB856000).

\end{document}